\documentclass[final]{siamltex}
\usepackage{amsmath,amsfonts,amssymb}
\usepackage{latexsym}
\usepackage{epsfig,overpic}
\usepackage{subfigure}
\usepackage{enumitem}
\usepackage{color}
\usepackage{hyperref}
\usepackage[normalem]{ulem}
\usepackage{showkeys}
\usepackage{showlabels}

\newcommand{\R}{\mathbb R}

\newcommand{\M}{\mathcal M}

\newcommand{\cH}{\mathcal H}
\newcommand{\cI}{\mathcal I}

\newcommand{\cM}{\mathcal M}
\newcommand{\cN}{\mathcal N}

\newcommand{\cG}{\mathcal G}
\newcommand{\cE}{\mathcal E}

\newcommand{\ssc}{{\sf c}}

\newcommand{\sC}{{\sf C}}

\newcommand{\fsas}{f_{\rm sas}}

\newcommand{\intset}{\mathrm{int}}
\newcommand{\bset}{\mathrm{b}}

\newcommand{\cf}{\gamma_{\mathrm{f}}}
\newcommand{\Nmax}{N_{\max}}
\newcommand{\Nzero}{N_0}
\newcommand{\gint}{\gamma_{\intset}}
\newcommand{\gb}{\gamma_{\bset}}
\newcommand{\ga}{\gamma_\alpha}

\newcommand{\Oie}{\Omega_i^{\rm e}}
\newcommand{\Gie}{\Gamma_i^{\rm e}}

\newcommand{\rp}{r_{\mathrm{p}}}

\DeclareMathOperator*{\argmin}{arg\,min}

\DeclareGraphicsExtensions{.pdf,.eps,.ps,.eps.gz,.ps.gz,.eps.Y}

\newtheorem{assumption}{Assumption}[section]
\newtheorem{indicator}{Indicator}[section]

\newtheorem{remark}{Remark}[section]

\newcommand{\rev}[1]{#1}

\newcommand{\black}{\color{black}}
\newcommand{\blue}{\color{blue}}

\begin{document}
\title{Analysis of the Schwarz domain decomposition method for the conductor-like screening continuum model}
\author{Arnold Reusken \and Benjamin Stamm}
\maketitle

\begin{abstract}
We study the Schwarz overlapping domain decomposition method applied to the Poisson problem on a special family of domains, which by construction consist of a union of a large number of fixed-size subdomains. These domains are motivated by applications in computational chemistry where the subdomains consist of van der Waals balls. 
As is usual in the theory of domain decomposition methods, the rate of convergence of the Schwarz method is related to a stable subspace decomposition. We derive such a stable decomposition for this family of domains and analyze how the stability ``constant'' depends on relevant geometric properties of the domain. For this, we introduce new descriptors that are used to formalize the geometry for the family of domains. We show how, for an increasing  number of subdomains, the rate of convergence of the Schwarz method depends on specific local geometry descriptors and on one global geometry descriptor.  
The analysis also naturally provides lower bounds in terms of the descriptors for the smallest eigenvalue of the Laplace eigenvalue problem for this family of domains.
\end{abstract}

\section{Introduction}
In this article, we analyze scaling properties of the Schwarz overlapping domain decomposition method for the Poisson problem: find $u\in H_0^1(\Omega^M)$ such that
\[
	-\Delta u = f, \qquad\mbox{in } H^{-1}(\Omega^M):=\big(H_0^1(\Omega^M)\big)'.
\]
Here $\Omega^M \subset \Bbb{R}^3$ is a bounded  Lipschitz domain and $H_0^1(\Omega^M)$ denotes the usual Sobolev space of functions with weak derivatives in $L^2(\Omega^M)$ with vanishing Dirichlet-trace.
We investigate the behavior of the Schwarz iterative method when $\Omega^M$ consists of a increasing number $M=2,3,\ldots$ of {\it fixed-size} overlapping subdomains $\{\Omega_i\}_{i=1}^M$. We are particularly interested in the case that the subdomains $\Omega_i$ are overlapping balls with comparable radii. 

The motivation for studying this problem comes from numerical simulations in computational chemistry. Recently, a domain decomposition method has been proposed~   \cite{ddCOSMO1,ddCOSMO3,ddCOSMO4,ddCOSMO2,ddCOSMO5} in the context of so-called implicit solvation models, more precisely for the COnductor-like Screening MOdel (COSMO)~\cite{klamt1993cosmo} which is a particular type of  continuum solvation model (CSM).
In a nutshell, such models account for the mutual polarization between a solvent, described by an infinite continuum, and a charge distribution of a given solute molecule of interest. It therefore takes the long-range polarization response of the environment (solvent) into account. We refer to the review articles~\cite{lipparini2016perspective,tomasi2005quantum} for a thorough introduction to continuum solvation models. 

While for most of the applications of domain decomposition methods, the computational domain remains fixed (such as in engineering-like applications) and finer and finer meshes are considered, applications in the present context deal with different molecules consisting of a (very) large  number of atoms.
Each atom is associated with a corresponding van der Waals (vdW)-ball with a given and element-specific radius so that the total computational domain consists of the union of those vdW-balls. 
For a set of different molecules the computational domain is therefore changing and the Schwarz domain decomposition exhibits different convergence properties. 
For example, for an (artificial) linear chain of atoms of increasing length the Schwarz domain decomposition is scalable and does not require a so-called coarse space correction~\rev{\cite{ddCOSMO1}.}
 
A general convergence analysis of the Schwarz domain decomposition iterative method for the \emph{family} of domains $\Omega^M$, $M=2,3,\ldots$, can not easily be deduced  from classical analyses of domain decomposition methods available in the literature, e.g. \cite{Dolean,Toselli,Xuoverview}. This is due to the fact that these classical analyses assume a \emph{fixed domain} that is decomposed in an increasing number of (overlapping) subdomains of decreasing size, whereas in the setting outlined above the subdomains all have a given comparable size and the \emph{global domain changes} when the number of subdomains is increased. It turns out that for a convergence analysis in the latter case it is not obvious how results and tools available in classical analyses can be applied. Therefore, in recent papers \cite{chaouqui2018scalability,ciaramella2017analysis,ciaramella2018analysis,ciaramella2018analysisb,ciaramella2019scalability}
this topic has been addressed and new results on the convergence of the Schwarz domain decomposition iterative method on a family of domains $\Omega^M$, $M=2,3,\ldots$, were obtained. 
More precisely, the first theoretical results were obtained for a chain to rectangles in two dimensions~\cite{chaouqui2018scalability,ciaramella2017analysis}, which were later generalized to chain-like structures of disks and balls in two respectively three dimensions~\cite{ciaramella2018analysis,ciaramella2018analysisb}. 
These results, however, cover only (very) special cases as each of the subdomains has a nonempty intersection with the boundary of the computational domain $\Omega^M$, i.e. no balls are allowed that are contained inside $\Omega^M$.
A first step towards a more general analysis can be found in~\cite{ciaramella2019scalability}, which analyzes how  the error propagates and contracts in the maximum norm in a general geometry. 
It is shown that for a molecule with $N$-layers, it takes $N+1$ iterations until the first contraction in the maximum norm is obtained.
This is essential to understand the contraction mechanism, in particular for the first iterations, but unfortunately does not provide much insight on the rate of (asymptotic) convergence.

In this paper we present a  general analysis which covers many cases that occur in applications and that goes beyond the previously mentioned contributions. Although the presentation is somewhat technical, due to the fact that we have to formalize the geometry of the family of domains $\Omega^M$, $M=2,3,\ldots$,  the convergence analysis is based on  a few fundamental ingredients known from the field of subspace correction methods and Sobolev spaces, which are combined with new descriptors of the geometries considered. We outline the main components of the analysis.  We use the well-established framework of subspace correction methods \cite{Xuoverview}. In \cite{Xu}, for the successive (also called ``multiplicative'') variant of the Schwarz domain decomposition method a  convergence analysis in a general Hilbert space setting is derived. The contraction number of the error propagation operator (in the natural energy norm) can be expressed in only \emph{one stability parameter} ($s_0$ in 
Lemma~\ref{lemmaXu} below). This parameter quantifies the \emph{stability of the space decomposition}. 
We bound this stability parameter by introducing a  \emph{new variant of the pointwise Hardy inequality}. This variant allows estimates that take  certain important global geometry properties into account. Using this we derive, for example,  a uniform  bound in $M$, if we have a chain like family of domains, and a bound that grows (in a specified way) as a function of $M$, if we have a family of ``globular '' domains. It is well-known from the literature on domain decomposition methods that in  the latter case one should use an additional ``global coarse level space''. We  propose such a space for our setting and analyze the rate of convergence of the Schwarz method that includes this additional coarse space.  


The paper is organized as follows. In Section~\ref{Sectproblem} the Schwarz domain decomposition that we analyze in this paper is explained and an important result on the rate of convergence of this method, known from the literature, is given. This result essentially states that the contraction number (in the energy norm) of the Schwarz method is characterized by only \emph{one} quantity, which controls the stability of the space decomposition. In Section~\ref{Sectgeometry} we introduce new descriptors of the specific class of domains (union of overlapping balls) that is relevant for our applications. Furthermore, for this class of domains a natural partition of unity is defined and analyzed. This partition of unity is used in Section~\ref{SectAnalysis} to derive bounds for the stability quantity. A further key ingredient in our analysis of the Schwarz method is a variant of the pointwise Hardy inequality, that is also presented in Section~\ref{SectAnalysis}. A main result of the paper is given 
Theorem~\ref{thmstab}. As is well-known from the theory of Schwarz domain decomposition methods, in certain situations the efficiency of such a method can be significantly improved by using a global (coarse level) space. For our particular application this issue is studied in Section~\ref{sectcoarse}. Finally, in Section~\ref{Sectexperiments} we present results of numerical experiments, which illustrate certain properties of the Schwarz domain decomposition method and relate these to the results of the convergence analysis.

\section{Problem formulation and Schwarz domain decomposition method} \label{Sectproblem}
We first describe the class of domains $\Omega^M$ that we consider. Let $m_i \in \Bbb{R}^3$, $i=1,\ldots,M$, be the centers of balls and $R_i$ the corresponding radii. 
We define 
\begin{align*}
  \Omega_i &:= B(m_i;R_i) = \{ \, x \in \Bbb{R}^3~|~\|x-m_i\| < R_i \,\}, \\
  \Omega^M & := \cup_{i=1}^M \Omega_i.
\end{align*}
We consider  the  Poisson equation: determine $u \in H_0^1(\Omega^M)$ such that
\begin{equation} 
	\label{problemdef}
	 a(u,v):=\int_{\Omega^M} \nabla u \cdot \nabla v \, dx = f(v) \qquad \text{for all}~~v \in H_0^1(\Omega^M),
\end{equation}
with a given source term $f \in H^{-1}(\Omega^M)$. 
For a subdomain $\omega \subset \Omega_M$ we denote the Sobolev seminorm of first derivatives by $|v|_{1,\omega}^2:= \int_\omega \|\nabla v(x)\|^2 \, dx$.
For solving this problem we use the Schwarz domain decomposition method, also called successive subspace correction in the framework of Xu and Zikatanov~\cite{Xu}. This method is as follows:
\begin{flalign}
  ~~~& \text{Let}~~u^0 \in H_0^1(\Omega^M)~~\text{be given}. &&\nonumber\\\nonumber
  ~~~& \text{for}~~\ell =1,2, \ldots \\\nonumber
  ~~~&~~~~u_0^{\ell-1}:=u^{\ell-1} \\\nonumber
  ~~~&~~~~\text{for}~~i=1:M \\\nonumber
  ~~~&~~~~~~~~\text{Let}~~e_i \in H_0^1(\Omega_i)~~\text{solve}\\
  \label{SSC}
  ~~~&~~~~~~~~~~~~a(e_i,v_i)=f(v_i)-a(u_{i-1}^{\ell-1},v_i) \quad \text{for all}~~v_i \in H_0^1(\Omega_i). \\\nonumber
  ~~~&~~~~~~~~u_i^{\ell-1}:=u_{i-1}^{\ell-1}+e_i \\\nonumber
  ~~~&~~~~\text{endfor} \\\nonumber
  ~~~&~~~~u^\ell:=u_M^{\ell-1}\\\nonumber
  ~~~& \text{endfor}
\end{flalign}
This is a linear iterative method and its error propagation operator is denoted by $E$. We then obtain
\[
 u-u^{\ell}=E(u-u^{\ell-1})= \; \ldots \; =E^{\ell}(u-u^0).
\]
An analysis of this method is presented in an abstract Hilbert space framework in~\cite{Xu}. Here we consider only the case, in which the subspace problems in \eqref{SSC} \emph{are solved exactly}. 
\begin{remark} 
 \rm
 There also is an additive variant of this subspace correction method (called ``parallel supspace correction'' in \cite{Xu}). This method may be of interest because it has much better parallelization properties. This additive variant can be analyzed with tools very similar to the successive one given above. We further discuss  this in Remark~\ref{Remadditive}.
\end{remark}

As norm on $H_0^1(\Omega^M)$ it is convenient to use $|v|_{1,\Omega^M}^2:=a(v,v)$. In this norm the bilinear form $a(\cdot,\cdot)$
has ellipticity and continuity constants both equal to 1. The corresponding operator norm on $H_0^1(\Omega^M)$ is also denoted by $|\cdot|_{1,\Omega^M}$. On $H^1(\Omega_i)$ we use the seminorm $|v_i|_{1,\Omega_i}^2:=a(v_i,v_i)$, $v_i \in H^1(\Omega_i)$. We need the following projection operator $P_i: H_0^1(\Omega^M) \to H_0^1(\Omega_i)$ defined by
\[
  a(P_iv, w_i)=a(v,w_i)\qquad \text{for all}~~w_i \in H_0^1(\Omega_i).
\]

We recall an important result from \cite{Xu} (Corollary 4.3 in \cite{Xu}).
\begin{lemma} \label{lemmaXu}
 Assume that $\sum_{i=1}^M H_0^1(\Omega_i)$ is closed in $H_0^1(\Omega^M)$. Define
 \begin{equation} \label{c0}
  s_0:= \sup_{\substack{v\in H_0^1(\Omega^M) \\ |v|_{1,\Omega^M}=1}} \inf_{\sum_{j=1}^M v_j =v} \sum_{i=1}^M \Big|P_i \sum_{j=i+1}^M v_j\Big|_{1,\Omega_i}^2,
 \end{equation}
where $v_j \in H_0^1(\Omega_j)$ for all $j$. Then
\begin{equation} \label{resXu}
  |E|_{1,\Omega^M}^2 = \frac{s_0}{1+s_0}
\end{equation}
holds. 
\end{lemma}
\ \\

The constant $s_0$ quantifies the \emph{stability of the decomposition} of the space $H_0^1(\Omega^M)$ into the sum of subspaces $H_0^1(\Omega_i)$. Due to the result \eqref{resXu} we have that the contraction number of the Schwarz domain decomposition method (in the natural $|\cdot|_{1,\Omega^M}$ norm) depends only on $s_0$. Hence, if $s_0$ is independent of certain parameters (e.g., in our setting $M$) then the contraction rate is also robust w.r.t. these parameters. In the remainder of this paper we analyze this stability quantity $s_0$ depending on the geometrical setting and the closedness assumption needed in Lemma~\ref{lemmaXu}.
The analysis is based on a particular decomposition $v= \sum_{i=1}^M \theta_i v=  \sum_{i=1}^M v_i$ with $v_i:=\theta_i v$ and $(\theta_i)_{1 \leq i \leq M}$ forms a partition of unity that is introduced and analyzed in the next section.

\section{Geometric properties of the domain $\Omega^M$ and partition of unity} \label{Sectgeometry}

The number $M$ of balls is arbitrary and in the analysis below it is important that in estimates and in further results we explicitly address the dependence on the number~$M$. The estimates in the analysis below depend on certain geometry related quantities that we introduce in this section.

In order to formalize the geometry dependence in our estimates,  we introduce a (infinite but countable) family of geometries $\{\mathcal F_M\}_M$ indexed by the increasing number $M \in \Bbb{N}$ of balls, where each element $\mathcal F_M = \{\,B(m_i,R_i) \; | \; i=1,\ldots M \,\}$ represents the set of balls defining the geometry $\Omega^M$ characterized by the set of centers and radii. 
We further introduce
\[
	R_{\min} (\mathcal F_M):=\min_{1 \leq i \leq M} R_i, 
	\qquad 
	R_{\max} (\mathcal F_M):=\max_{1 \leq i \leq M} R_i.
\]

 We first start with stating basic assumptions on the geometric structure of the considered domains.
\begin{assumption}[Geometry assumptions]\rm
\begin{description}
[leftmargin=1cm,style=nextline,itemsep=5pt,font=\mdseries,topsep=3pt]
\item[{\bf (A1)}] For each $M$, we assume that $\Omega^M$ is connected. This assumption is made without loss of generality. If $\Omega^M$ has multiple components, the problem \eqref{problemdef} decouples into independent problems on each of these components and the analysis presented below applies to the problem on each  component.
\item[{\bf (A2)}] For each $M$, we assume that there are no $i,j$, with $i\neq j$, such that $\Omega_i \subset \Omega_j$, i.e., balls are not completely contained in larger ones. Otherwise, the inner balls can be removed from the geometric description without further consequences.
\item[{\bf (A3)}] 
We assume the radii of the balls to be uniformly bounded in the family $\{\mathcal F_M\}_M$: there exists $ R_{\max}^\infty$ and $R_{\min}^\infty > 0$ such that
\[
  \sup_M R_{\max} (\mathcal F_M) \le R_{\max}^\infty < \infty, \qquad \inf_M R_{\min} (\mathcal F_M) \ge R_{\min}^\infty > 0.
\]
\item[{\bf (A4)}]
(Exterior cone condition) We assume that 
 for each $y\in\partial\Omega^M$ there exists a circular cone $C(y)$ with \rev{aperture $2 \beta\ge 2 \beta^M>0$,} apex $y$ and axis $n(y)$ that belongs entirely to the outside of $\Omega^M$ in a neighborhood of $y$, i.e., $B(y;\epsilon) \cap C(y) \cap \Omega^M = \emptyset$ for $\epsilon >0$ sufficiently small.
We furthermore assume that  $\beta^M$ is uniformly bounded from below: there exists $\beta^\infty > 0$ such that
\[
	\inf_M \beta^M  \ge \beta^\infty > 0.
\]

\end{description}
\end{assumption}
\noindent
Related to this we have the following result.
\begin{lemma}
\label{lem:A4equiv}
For  $y\in\partial\Omega^M$  denote by $i_t$, $t=1,\ldots,r$, all indices such that 
$
	y\in \partial \Omega_{i_t},
$
and define $v_t = \frac{m_{i_t}-y}{\|m_{i_t}-y\|}$.
Assumption {\bf (A4)} is equivalent to the following one:

There exists  $\ga^\infty>0$ such that for each $M$ and for each $y\in\partial\Omega^M$, there exists a unit vector $n(y)$ such that $-n(y)\cdot v_t \ge \ga^\infty>0$ for all $t=1,\ldots,r$. This implies that all vectors $v_t$ are situated on one side of the plane that is perpendicular to $n(y)$ and passing through $y$. 
\end{lemma}
\begin{proof}
The limiting \rev{aperture} of a cone with apex $y$ in the direction of $n(y)$ is given by twice the minimal  angle of $n(y)$ with the tangential plane at $y$ to each ball $\Omega_{i_t}$, $t=1,\ldots,r$. This angle is illustrated by $\beta_{i_t}(y)$ in Figure~\ref{fig:IntExt}. Note that $\alpha_{i_t}+\beta_{i_t}=\tfrac12 \pi$, where $\alpha_{i_t}$ is the angle between $n(y)$ and $v_{i_t}$. Hence $\beta_{i_t}$ is bounded away from zero if and only if $\alpha_{i_t}$ is bounded away from $\tfrac12 \pi$, i.e. $ \cos (\alpha_{i_t})$ bounded away form zero. Finally note that $ \cos (\alpha_{i_t}) = - n(y) \cdot v_{i_t}$, which shows that the uniform boundedness away from zero of the interior cone angles $\beta_{i_t}$ and of $- n(y) \cdot v_{i_t}$ are equivalent conditions.
\end{proof}


\black
The condition of Lemma~\ref{lem:A4equiv} provides a precise mathematical statement in terms of geometrical notions. 
This condition for instance excludes the following scenarios, using the notation $D=\mbox{dim}(\mbox{span}(v_{i_1},\ldots,v_{i_r}))$, and where Figure~\ref{fig:IntExt} (middle and right) provides a schematic illustration of those two cases:
\begin{description}
\item[$D=1$:] Intersection of two  balls $\overline\Omega_{i_t}$ is only one point, that is, $y$ and the centers $m_{i_t}$ are on one line. In turn, only the plane $\mathcal P$ passing through $y$ which is perpendicular to line passing through $v_{i_1}$ and $v_{i_2}$ does not intersect $\Omega^M$ locally around $y$ and there exists no cone of positive \rev{aperture} with apex $y$ that (locally) belongs to the outside of $\Omega^M$.
\item[$D=2$:] Intersection of three  balls $\overline\Omega_{i_t}$ is one point. Here, only the line $y+ tw$, with $w\in\R^3$ being the normal vector to the plane passing through $v_{i_1}$, $v_{i_2}$, $v_{i_3}$ and $t\in(-\varepsilon,\varepsilon)$, belongs to the outside of $\Omega^M$. In turn, there exists no cone of positive \rev{aperture} with apex $y$ that (locally) belongs to the outside of $\Omega^M$.
\end{description}

\begin{figure}[t!]
	\centering
	\includegraphics[trim=50pt 120pt 420pt 80pt, clip, scale=0.25]{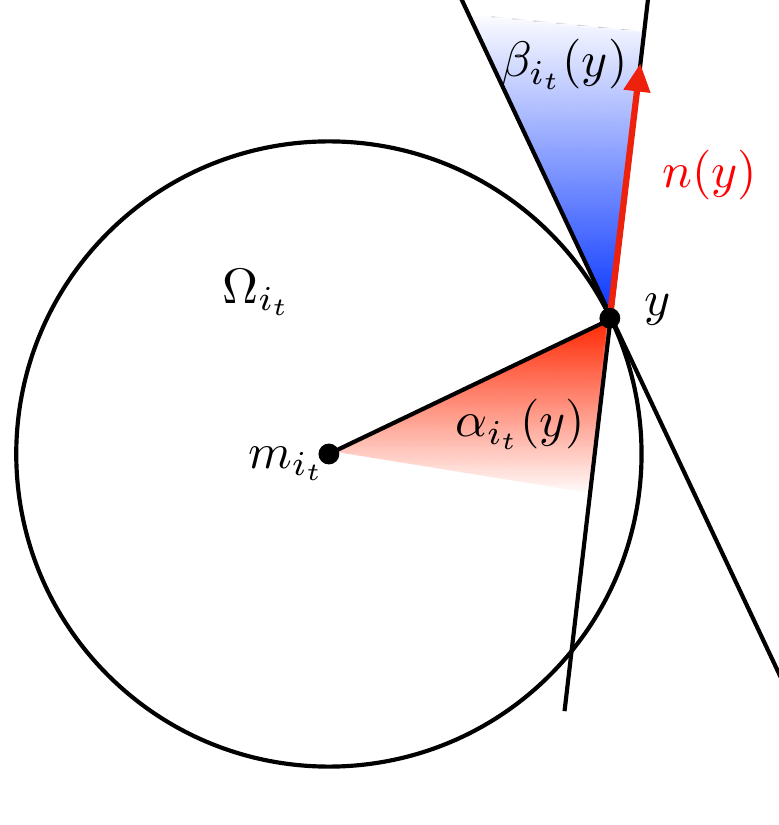}
	\includegraphics[trim=50pt 580pt 50pt 40pt, clip, scale=0.45]{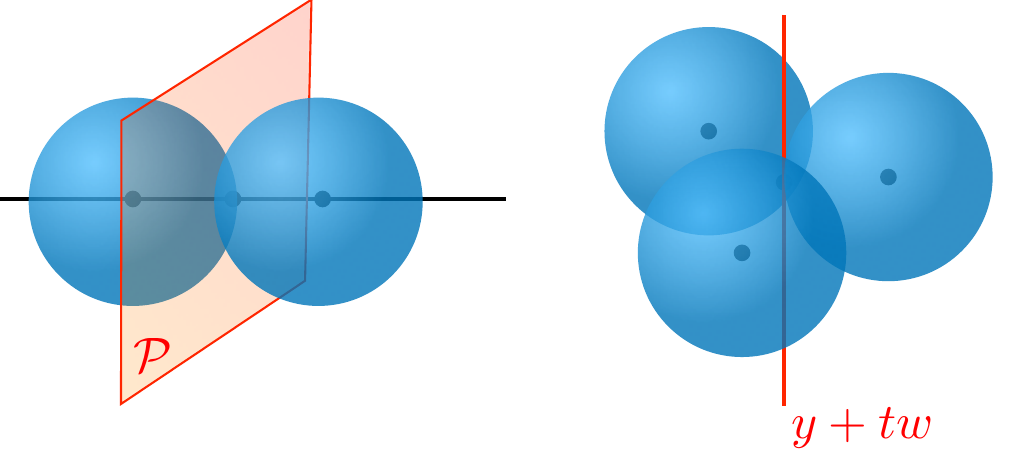}
\caption{Relation between $\alpha$ and $\beta$ (left), illustrative example for the case $D=1$ (middle) and for the case $D=2$ (right) related to the violation of condition {\bf (A4)}.}
	\label{fig:IntExt}
\end{figure}

\black 

\subsection{Local geometry indicators}
We introduce certain geometry descriptors, which we call \emph{indicators}, that are related to the specific geometry of the domain $\Omega^M$ and that will be used in the estimates derived below.
We will distinguish between local and global indicators, the former only being dependent on \emph{local geometrical features} whereas the latter being dependent on the \emph{global topology} of the geometric configuration.

We introduce some further definitions. We take a fixed $\mathcal F_M \in \{\mathcal F_M\}_M$, with corresponding domain  $\Omega^M$.
We decompose the index set $\cI:=\{1,\ldots,M\}$ into two disjoint sets by introducing $\cI_{\intset}:= \{\, i \in \cI~|~\partial\Omega_i \cap \partial \Omega^M= \emptyset \,\}$ (``interior balls'') and $\cI_{\bset}:=\cI  \setminus \cI_{\intset}$ (``boundary balls''). The corresponding (overlapping) subdomains are denoted by $\Omega_{\intset}:= \cup_{i \in \cI_{\intset}} \Omega_i$, $\Omega_{\bset}:= \cup_{i \in I_{\bset}} \Omega_i$. It may be that $\cI_{\intset}$ is an empty set. We define, for $i \in \cI$,  $\cN_i:=   \{\, j \in \cI~|~\Omega_i \cap \Omega_j\neq \emptyset\,\}$, $\cN_i^0:= \cN_i \setminus \{i\}$. 
For $i \in \cI_{\bset}$ we define 
$\Gamma_i := \partial\Omega_i \cap \partial \Omega^M$.
Further, define
\begin{align*}
  \delta_i(x) & := \max\{0, R_i- \|x- m_i\|\,\},  &&x \in \overline\Omega^M,~~i\in \cI,\\
  \delta(x) & :=\sum_{i \in \cI} \delta_i(x),  &&x \in \overline\Omega^M.
\end{align*}
Hence, on $\Omega_i$ the function $\delta_i$ is the distance function to $\partial \Omega_i$, which is extended by 0 outside $\Omega_i$. Note that $\delta(x)= \sum_{j \in \cN_i} \delta_j(x) $ for $x \in \Omega_i$  and that, for any $x
\in\Omega^M$, there holds that $\delta(x)=0$ if and only if $x\in\partial\Omega^M$. 

In the following, we list the indicators that are used in the upcoming analysis.
\begin{indicator}[Maximal number of neighbors]
\rm
Define 
\begin{equation} 
	\label{card2}
	\Nmax := \max_{i\in \cI} \, \mbox{card}(\cN_i ).
\end{equation}
i.e., $\Nmax-1$ is the maximal number of neighboring balls that overlap any given ball~$\Omega_i$. \rev{Note, that if to each $\Omega_i$ we associate a number $\alpha_i \geq 0$, then $\sum_{i=1}^M\sum_{j \in \cN_i} \alpha_j \leq \Nmax \sum_{i=1}^M \alpha_i$ holds.}
\end{indicator}

\begin{indicator}[Maximal overlap indicator]\rm
Let  $\Nzero$ be the smallest integer such that:
\begin{equation} \label{card}
	\max_{x\in \Omega^M}{\rm card}\{\, j~|~x\in \Omega_j\,\} \leq \Nzero. 
\end{equation}
Hence, $\Nzero-1$ is the maximal number of neighboring balls that overlap any given point $x\in \Omega_i$ of any given ball $\Omega_i$.
\end{indicator}

\rev{
Note that for $i \in \cI_{\intset}$ we have $\partial \Omega_i \subset \cup_{j \in \cN_i^0} \Omega_j$, hence $\delta(x) >0$ for all $x \in \partial \Omega_i$,  and thus  $\delta(x) >0$ for all $x \in \Omega_i$.   
In the same vein, we recall that for all $x\in\Omega^M$ such that $\mbox{dist}(x,\partial \Omega^M) > 0$, there holds $\delta(x)>0$.
This motivates us to define the following indicator.
\begin{indicator}[Stable overlap indicator] 
\label{indintballs} 
\rm 
We define 
\begin{equation}
	\gint := \min_{x\in\Omega_{\intset,\beta}} \delta(x) >0,
\end{equation}
where 
\begin{equation}
	\Omega_{\intset,\beta} := \left\{ x\in \Omega^M | x\in\Omega_{\intset} \mbox{ or dist}(x,\partial \Omega^M) > R^\infty_{\min} \sin(\beta^\infty) \right\}.
\end{equation}
\end{indicator}
With this definition, there holds
 \begin{equation} \label{overlap1}
	\delta(x) \geq \gint, 
\end{equation}
for all $x \in \Omega_{\intset}$ or such that dist$(x,\partial \Omega^M) > R^\infty_{\min} \sin(\beta^\infty)$.
}
Note that by construction, this is a local indicator.
\begin{remark} \rm
The indicator $\gint$ is a measure for the amount of overlap between any interior ball and its neighboring balls. 
The indicator is small if there exists a point $x \in \Omega_i$, with $i \in \cI_{\intset}$, that is simultaneously close to $\partial \Omega_i$ and to the boundary $\partial \Omega_j$ of all balls $\Omega_j$ with $x\in\Omega_j$. 
\end{remark}

A proof of the following lemma, giving rise to a further indicator, is given in Appendix~\ref{App1}.
\begin{lemma} [Stable overlap for boundary balls] 
\label{lem:DistEq}
Under Assumption {\bf (A4)}, there exists $\gamma_{\bset}>0$, such that
\begin{equation}
	\label{overlap3}
  	\delta(x) \geq \gamma_{\bset} \,{\rm dist}(x,  \partial \Omega^M)\quad \text{for all}~x \in \Omega_{\bset}.
\end{equation}
\end{lemma} 
\begin{indicator}[Stable overlap for boundary balls] 
\label{ind:DistEq}
The constant $\gamma_{\bset}>0$ defined in Lemma~\ref{lem:DistEq} is considered as a geometry indicator.
\end{indicator}


\medskip 
The indicator $\gamma_\bset$ employed in \eqref{overlap3} clearly is a local one. An explicit formula $\gamma_{\bset}=\gamma_{\bset}(R_{\rm min}, R_{\rm max}, \beta^\infty)$ is given in Eqn. \eqref{formg}.

The four indicators introduced above are all natural ones, which are directly related to the number of neighboring balls and the size of the overlap between neighboring balls.

We need one further local indicator, which needs some introduction. In the analysis of the Schwarz method we use a (natural) partition of unity, cf. Section~\ref{sectPU}. The gradient of some of these partition of unity functions is unbounded at $\partial \Omega^M$, where their growth behaves like $x \to ({\rm dist}( x,\partial \Omega^M))^{-1}$. To be able to handle this singular behavior, we need an integral Hardy estimate of the form
\[
   \left(\int_{\Omega_{\bset}} \left(\frac{u(x)}{ {\rm dist}(x,  \partial \Omega^M)}\right)^2\, dx\right)^\frac12 \leq c\, |u|_{1,\Omega^M} \qquad \text{for all}~u \in H^1_0(\Omega^M),
\]
cf. Corollary~\ref{Coro1}. One established technique to derive such an estimate is as in e.g. \cite{Hajlasz,Kinnunen}, where \emph{pointwise} Hardy estimates are used to derive integral Hardy estimates. The analysis in this approach is based on a certain ``fatness assumption'' for the complement of the domain $\Omega^{M,c}:=\Bbb{R}^3 \setminus \Omega^M$. Here we follow this approach and below we will introduce a local indicator that quantifies this exterior fatness of the domain, which is very similar to the fatness indicator used in \cite{Hajlasz,Kinnunen} (cf., for example, Proposition 1 in \cite{Hajlasz}). Before we define the fatness indicator, note that due to the definition of $\Omega^M$ as a union of balls and assumption {\bf (A4)}, we have the following property
\begin{equation} \label{propfat1}
 \forall \,y \in \partial \Omega^M: ~~\exists r_0>0,~c>0:~~|B(y;r)\cap \Omega^{M,c}| \geq c\, |B(y;r)|\qquad \forall \,r \in (0,r_0].
\end{equation}  
\begin{indicator}[Local exterior fatness indicator] \label{indic45}\rm
For $ i \in \cI_{\bset}$ and $x \in \Omega_i$ we define a closest point projection on $ \Gamma_i$ by $p(x)$, i.e.,  $p(x) \in  \Gamma_i$ and $\|p(x)-x\|={\rm dist}(x,  \Gamma_i)$.  From \eqref{propfat1} it follows that there exists  $\hat c_i >0$ (depending on the constants $c=c(y)$ in \eqref{propfat1} and possibly also on $R_i$) such that
\begin{equation} \label{fatness1}
	 |B\big(p(x); \|p(x)-x\|\big)\cap \Omega^{M,c}| \geq \hat c_i |B\big(p(x),\|p(x)-x\|\big)|, \qquad \text{for all}~~x \in \Omega_i,
\end{equation}
with $ \Omega^{M,c}:= \Bbb{R}^3 \setminus \Omega^M$.
We define 
\begin{equation} \label{fatness}
	\cf:= \min_{i\in \cI_{\bset}} \hat c_i > 0 .
\end{equation}
\end{indicator}

\begin{remark}
	\label{rem:cones} \rm
	We call this a \emph{local} exterior fatness indicator because $\hat c_i$ depends only on a small neighbourhood of $\Omega_i$, consisting of points that have distance at most $R_i$ to $\Omega_i$. 
	The quantity $\hat c_i$ essentially (only) depends on two geometric parameters related  to exterior cones with apex at $y \in (\partial \Omega_i \cap\partial \Omega^M)$, namely the maximum possible aperture  and the maximal cone height such that the cone is completely contained in $\Omega^{M,c}$.
	For points lying only on one sphere $\partial\Omega_i$, the cone can be chosen to be as wide as a flat plane. 
	For points lying on an intersection arc $\partial\Omega_i\cap\partial\Omega_j$, the largest aperture of the cone is determined only by the center and radii of the two balls $\Omega_i, \Omega_j$. Finally any point lying on an intersecting point of three or more boundary spheres can be assigned a cone whose maximal aperture depends on the radii and centers of the associated balls.
	Figure~\ref{fig:ExtFat_cone} (left) provides a schematic illustration. The maximal cone height at $y \in (\partial \Omega_i\cap\partial \Omega^M)$ is related to the width of $\Omega^{M,c}$ at $y$ in the direction of the axes of the cone, see also Figure~\ref{fig:ExtFat_cone} (right) for a schematic 2D-illustration. If these apertures and heights are bounded away from zero (uniformly in $y \in (\partial \Omega_i\cap\partial \Omega^M)$), the quantity  $\hat c_i$ is bounded away from zero. In our applications we consider domains $\Omega^M$  such that the apertures and heights satisfy this property.
\end{remark}

\begin{figure}[t!]
	\centering
	\includegraphics[trim=280pt 200pt 250pt 160pt, clip, scale=0.18]{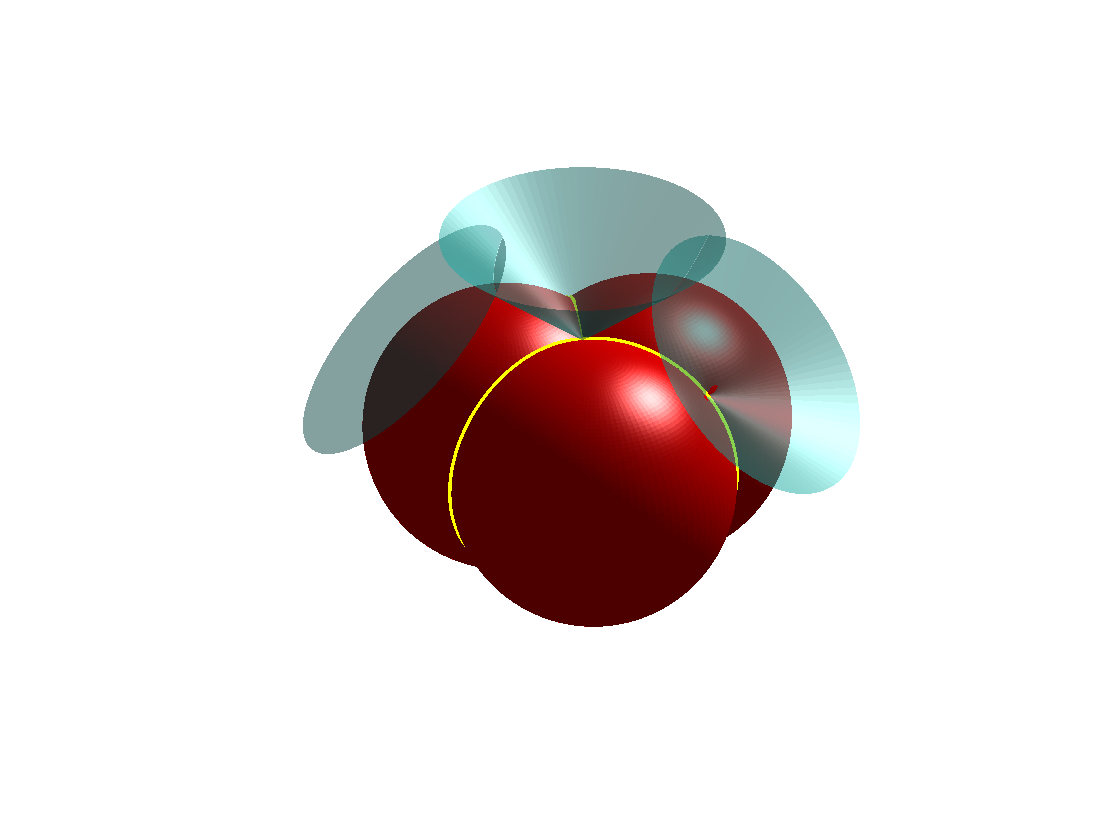}
	\includegraphics[trim=0pt 0pt 0pt 0pt, clip, scale=0.5]{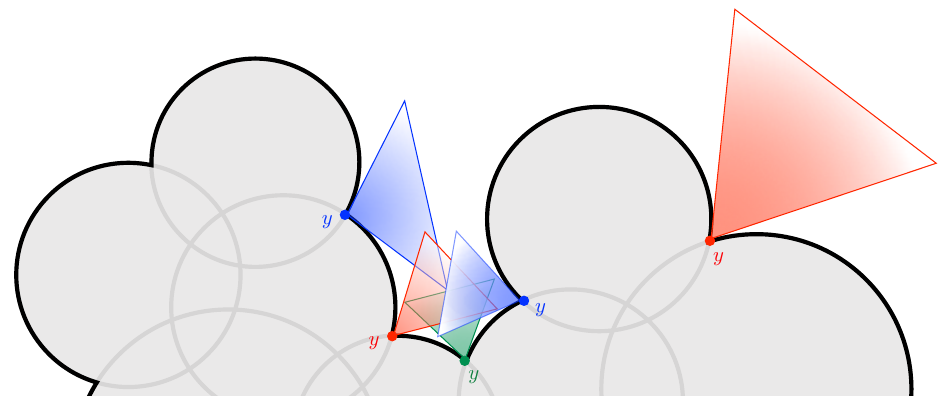}
\caption{Inserting cones of maximal aperture at the boundary points (left). Illustration of some inserted cones in the definition of the global fatness property (right).}
	\label{fig:ExtFat_cone}
\end{figure}

The set of \emph{local geometry indicators}  is denoted by
\begin{equation} \label{locgeo} 
	\cG_L^M:=\{\Nmax,\Nzero,\gint^{-1}, \gb^{-1},\cf^{-1}\}.
\end{equation}
We emphasize  that $\cG_L^M$ depends only on \emph{local geometry properties} of the domain (as explained above) and does not depend on the global topology of $\Omega^M$ (e.g., not relevant whether $\Omega^M$ is a linear chain of balls or has a globular form for example).  
We thus assume the following assumption.
\begin{assumption}[Asymptotic geometry assumption]\rm
\begin{description}
[leftmargin=1cm,style=nextline,itemsep=5pt,font=\mdseries,topsep=3pt]
\item[{\bf (A5)}] We assume that the local geometry indicators $\cG_L^M$ are uniformly bounded in the family $\{\mathcal F_M\}_M$.
\end{description}
\end{assumption}

\subsection{Global geometry indicator}
\label{ssec:GlobalIndicator}
In the analysis below we need a  Poincar\'e-Friedrichs inequality  $\|u\|_{L^2(\Omega^M)} \leq c \, |u|_{1,\Omega^M}$ for $u \in H^1_0(\Omega^M)$, cf.~Lemma~\ref{Lemma5}. As is well-known, the constant $c$ in  this inequality depends on \emph{global} geometry properties of the domain $\Omega^M$ and is directly related to the smallest Laplace eigenvalue in $H^1_0(\Omega^M)$. To control this constant we use an approach, presented in Section~\ref{sectHardy} below, based on pointwise Hardy estimates. For this approach to work one  needs a measure for ``global exterior fatness''. This measure resembles the one used in Indicator~\ref{indic45}, but there are two important differences. Firstly, we now consider $x \in \Omega^M$ instead of only $x \in \Omega_b$. Secondly,  for $x \in \Omega^M$, instead of the corresponding closest point projection $p(x)$ (used in Indicator~\ref{indic45}) we now take a possibly different  exterior point $b(x) \in \Omega^{M,c}=\Bbb{R}^3 \setminus \Omega^M$ such that with 
$d(x):=\|x-b(x)\|$ the exterior 
volume $\big|B\big(b(x);d(x)\big) \cap \Omega^{M,c} \big|$ is comparable to the volume $\big|B\big(b(x);d(x)\big)\big|$. The latter property is a key ingredient in the derivation of satisfactory pointwise Hardy estimates. It turns out that taking $b(x)=p(x)$ is not satisfactory. Below, we present a construction of ``reasonable'' points $b(x)$ that is adapted to the special class of domains that we consider, \rev{namely a union of balls. In the field of applications that we consider, this domain corresponds to a ``solute molecule'' that is surrounded by ``solvent molecules''. In this setting the so-called Solvent Accessible Surface (SAS) of the solute molecule ($\Omega^M$) is defined by the center of a ball (``an idealized spherical probe'') rolling over the solute molecule, that is, the surface enclosing the region in which the center of the ball can not enter.  The  Solvent Excluded Surface (SES), defined by the same spherical probe, is the surface enclosing the region that can be accessed by this spherical probe \cite{lee1971interpretation,rechards1977,connolly1983analytical}. Below in our construction we use these SAS-SES notions which are very natural for our class of domains. 
Explanations and further properties can be found in \cite{Quan}. Here we give only a few definitions and properties that are relevant for our analysis.}   
\begin{figure}[t!]
	\centering
	\includegraphics[trim=60pt 20pt 70pt 50pt, clip, scale=0.25]{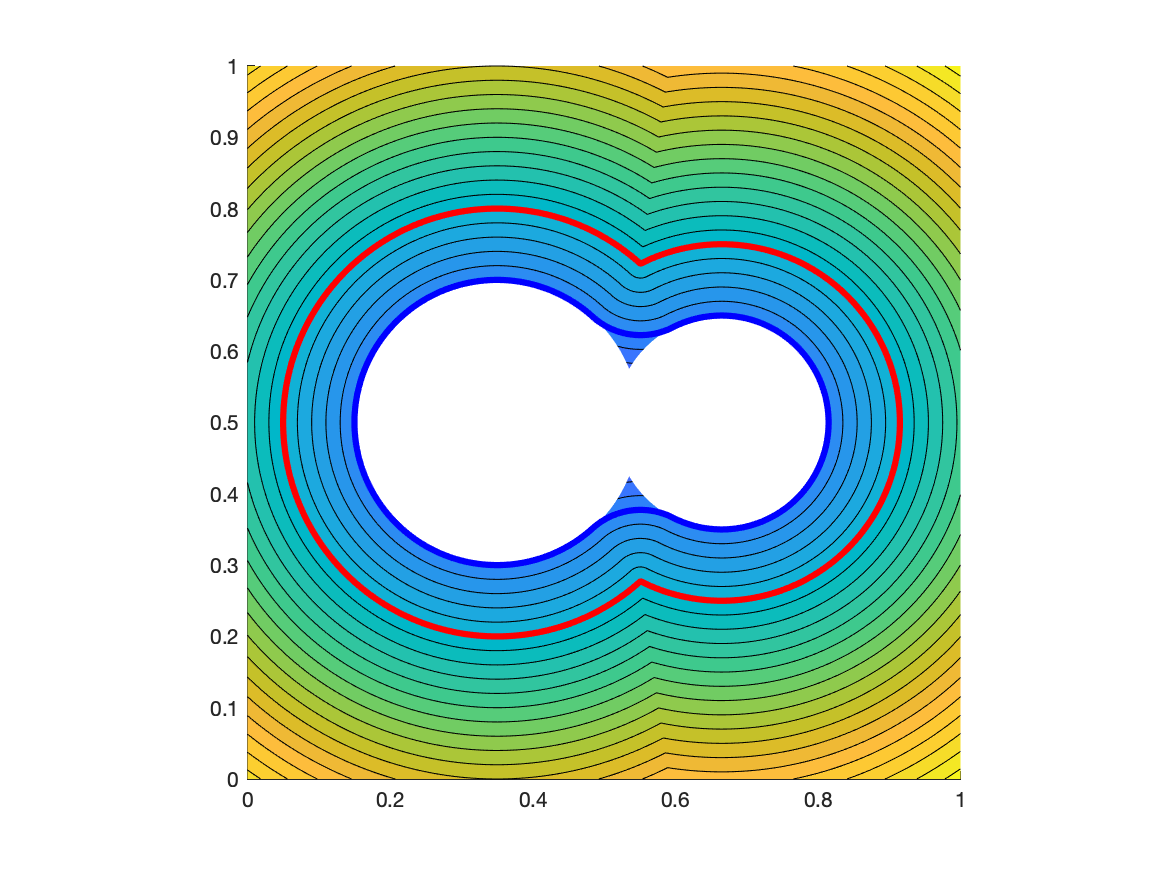}
	\includegraphics[trim=60pt 20pt 70pt 50pt, clip, scale=0.25]{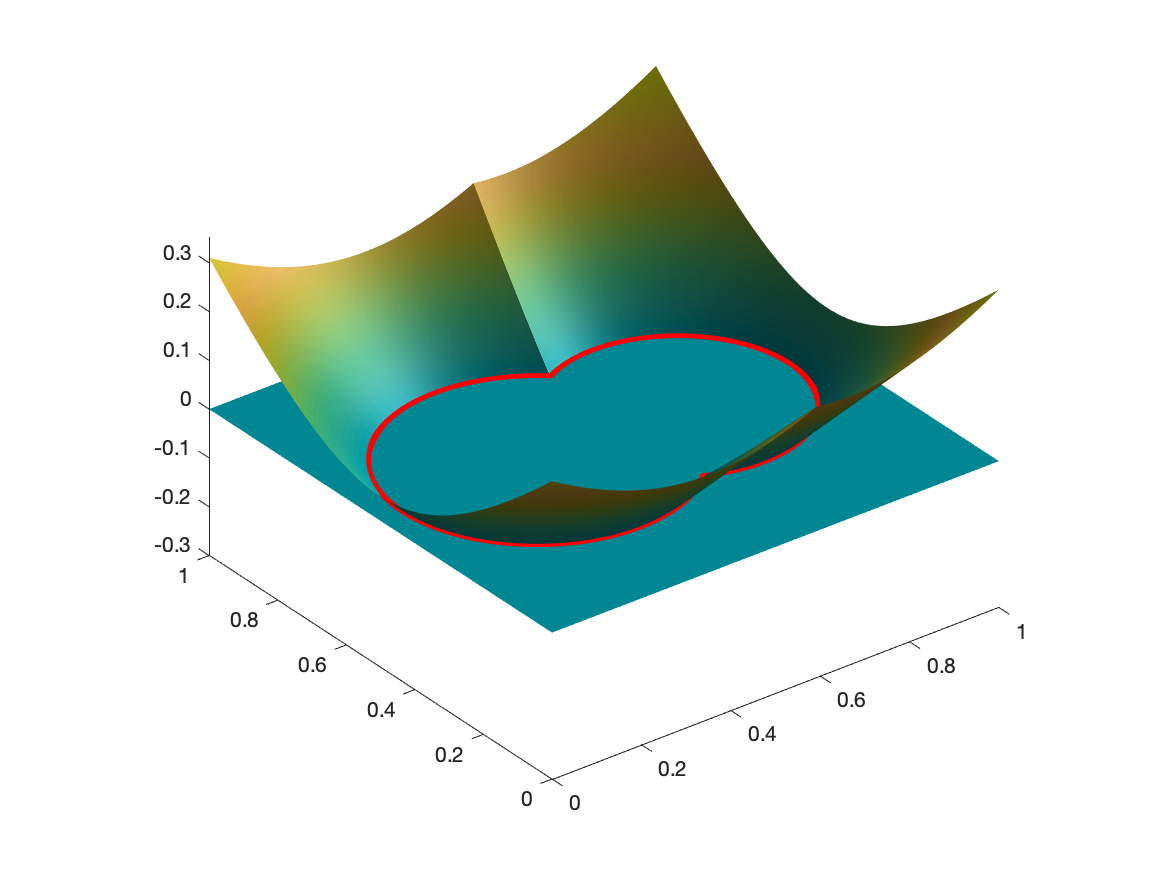}
	\includegraphics[trim=60pt 20pt 70pt 50pt, clip, scale=0.25]{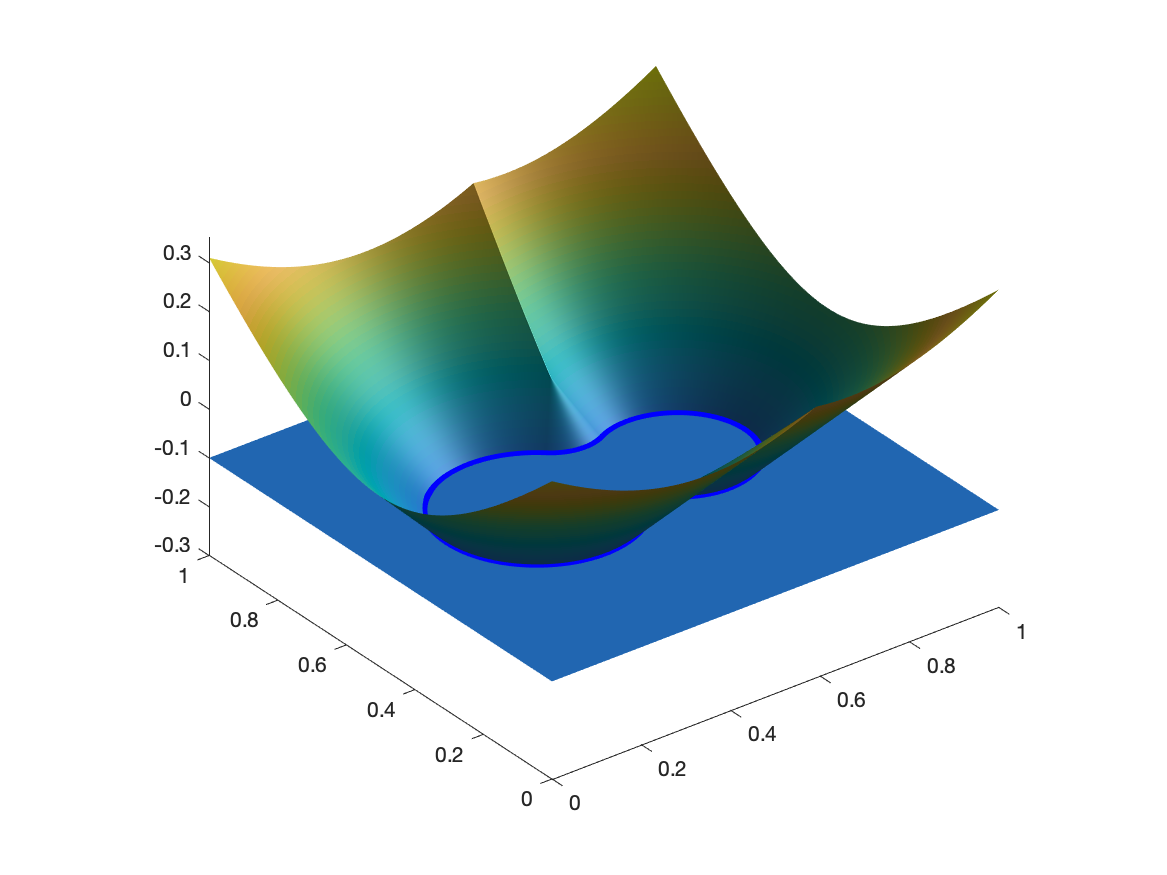}
\caption{Illustration of the signed distance function to the SAS and its level sets that define the SAS (red) and the SES (blue). The original domain $\Omega^M$ is illustrated in white. The distance between the two isolines is the so-called ``probe radius''.}
	\label{fig:SAS}
\end{figure}

\rev{For a given  (probe) radius $\rp>0$ the precise definition of the SAS is as follows. We introduce} 
\begin{align*}
  \Omega_{i,{\rm sas}}(\rp) &:= B(m_i;R_i+\rp) = \{ \, x \in \Bbb{R}^3~|~\|x-m_i\| < R_i + \rp \,\} \\
  \Omega^M_{\rm sas}(\rp) & := \cup_{i=1}^M \Omega_{i,{\rm sas}}(\rp),
\end{align*}
and define ${\rm SAS}:={\rm SAS}(\rp):=\partial  \Omega^M_{\rm sas}(\rp)$. Let $\fsas$ be the signed distance function to SAS (positive in $\Omega^M_{\rm sas}$), hence, ${\rm SAS}=\fsas^{-1}(0)$.  We define
\[
	{\rm SES}:={\rm SES}(\rp):=\fsas^{-1}(\rp) \subset \Omega^{M,c}.
\]
We refer to Figure~\ref{fig:SAS} for a graphical illustration of the definition of the SES and the SAS.
Further, we denote the maximal distance to the  $\partial \Omega^M$ by
\[
  d_{\Omega^M}:= \max_{x \in \Omega^M} {\rm dist}(x, \partial \Omega^M).
\]
A property of the SAS-SES construction is that the balls with center on ${\rm SAS}(\rp)$ and radius $r_p$ (these are tangent to ${\rm SES}(r_p)$)  are completely contained in $\Omega^{M,c}$. We will use these balls  in the construction of balls $B\big(b(x);d(x)\big)$, $x \in \Omega^M$, $b(x)\in \Omega^{M,c}$, $d(x):=\|x-b(x)\|$ that have ``sufficient exterior volume''. To determine a suitable $r_p$ we use $r_p=\lambda d_{\Omega^M}$, with $\lambda >0$ a parameter that will be specified below. We now explain this construction. 
A closest point projection on ${\rm SES}_\lambda:={\rm SES}(\lambda d_{\Omega^M})$ and the corresponding distance are denoted by
\[
  b_\lambda(x) \in  \arg \min_{y\in {\rm SES}_\lambda} \| x-y\|, 
  \qquad 
  d_\lambda(x) := \|b_\lambda(x)-x\|,
  \qquad 
  x \in \Omega^M. 
\]
The maximum distance to ${\rm SES}_\lambda$ is denoted by $d_{\rm SES}(\lambda):= \max_{x \in \Omega^M} d_\lambda(x)$.
Note that the following properties hold
\begin{equation}
	\label{resd1}
	\begin{aligned}[c]
	\lim_{\lambda \to 0^+}d_\lambda(x) &= {\rm dist}(x,\partial \Omega^M), \\
	\lim_{\lambda \to 0^+}d_{\rm SES}(\lambda) &=d_{\Omega^M}, 
	\end{aligned}
	\begin{aligned}[c]
	\quad 
	\lim_{\lambda \to\infty}d_\lambda (x) &= {\rm dist}(x,\partial\mbox{conv}(\Omega^M)), \\
	\quad 
	\lim_{\lambda \to\infty}d_{\rm SES}(\lambda) &= \max_{x \in \Omega^M} {\rm dist}(x,\partial\mbox{conv}(\Omega^M)).  
	\end{aligned}
\end{equation}
Define for all $x\in\Omega^M$, $B_0(x) := B(\hat b(x);\hat r(x))$,
with $\hat r(x): = \min(\lambda d_{\Omega^M}, \frac12 d_\lambda(x))$ and $\hat b(x) := b_\lambda(x) +\hat r(x) \frac{ b_\lambda(x)-x}{\|b_\lambda(x)-x\|}$.  

First, note that by the construction of the ${\rm SES}_\lambda$, there holds 
$
	B(\hat p(x) ;\lambda d_{\Omega^M})\subset\Omega^{M,c},
$
with $\hat p(x) := b_\lambda(x) + \lambda d_{\Omega^M}\frac{ b_\lambda(x)-x}{\|b_\lambda(x)-x\|} $,
i.e. this corresponds to the  \rev{definition of the SES of ``rolling a ball with  radius $r_p=\lambda \, d_{\Omega^M}$ and center on the SAS''. 
The point $\hat p(x) \in\,$SAS denotes this center of the ball.} 
See Figure~\ref{fig:ExtFat} for an illustration and~\cite{Quan} for further explanations.
Since $\hat r(x) \le \lambda d_{\Omega^M}$, there holds that for all $z\in B_0(x)$
\[
	\| z - \hat p(x) \| \le \| z - \hat b(x) \| + \| \hat b(x) - \hat p(x) \| 
	< \hat r(x) + (\lambda d_{\Omega^M} - \hat r(x)) = \lambda d_{\Omega^M},
\]
and thus
\[
	B_0(x)\subset B\left(\hat p(x);\lambda d_{\Omega^M}\right)\subset\Omega^{M,c}.
\]
Second, for any $z\in B_0(x)$, there holds
\[
	\| z - b_\lambda(x) \| \le \| z - \hat b(x) \| + \| \hat b(x) - b_\lambda(x) \| 
	< 2 \hat r(x) \le  d_\lambda(x),
\]
and thus 
\[	
	B_0(x)\subset B(b_\lambda(x);d_\lambda(x))=:B(x).
\]
Hence, for $x \in\Omega^M$ we have
\begin{align*}
	 \frac{|B(x)|}{|B_0(x)|}
 	=
  	\frac{d_\lambda(x)^3}{\hat r(x)^3} 
    =
  	8 \max\left\{ \left(\frac{d_\lambda(x)}{2\lambda d_{\Omega^M}}\right)^3, 1\right\} 
  	\leq  
  	8 \max\left\{ \left(\frac{d_{\rm SES}(\lambda)}{2 \lambda d_{\Omega^M}}\right)^3, 1\right\}.
\end{align*}
We define 
\begin{equation} \label{defGammal}
	\gamma_\lambda 
	=\frac18\min\left\{ \left(\frac{2\lambda d_{\Omega^M}}{d_{\rm SES}(\lambda)}\right)^3, 1\right\},
\end{equation}
and thus the exterior fatness estimate $|B_0(x)| \geq \gamma_\lambda |B(x)|$ for all $x \in \Omega^M$ holds. In the Hardy estimates used below (cf. \eqref{Hardy1}) we are interested in (posssibly small) bounds of the quotient $d_{\rm SES}(\lambda)/\gamma_\lambda$. This motivates the introduction of the following indicator:
\begin{indicator}[Global exterior fatness indicator] \label{ind46} 
\rev{Assume that} 
\begin{equation}
	\label{indDf1}
 	\lambda_{\min}:=
 	\argmin_{\lambda >0} \frac{ d_{\rm SES}(\lambda)}{\gamma_\lambda}
\end{equation}
\rev{exists, cf. Remark~\ref{rmexist}.} 
Define the corresponding {\rm global} fatness indicator
\begin{equation} 
	\label{indDf}
 d_F:= \frac{d_{\rm SES}(\lambda_{\min})}{\gamma_F}, \quad \text{with}~\gamma_F := \gamma_{\lambda_{\min}}.
\end{equation}

It follows that for any $x \in \Omega^M$, there exists a corresponding point $b(x) \in \Omega^{M,c}$ such that
 \begin{equation}
 	\label{eq:GlobFatness}
 	\big|B(b(x);\|b(x)-x\|) \cap \Omega^{M,c}\big| \geq \gamma_F  \big|B(b(x);\|b(x)-x\|)\big|.
 \end{equation}
\end{indicator}
\begin{remark} \label{rmexist} \rm The function $ \lambda \to  q(\lambda):=\frac{ d_{\rm SES}(\lambda)}{\gamma_\lambda}$ in \eqref{indDf1} is not necessarily continuous. Discontinuities can appear, for example, for $\lambda$ values at which holes in $\Omega_{\rm sas}^M(\lambda\, d_{\Omega^M})$ ``disappear''. \rev{In the generic case, however,  a (possibly non-unique) minimizer $\lambda_{\min}$ as in \eqref{indDf1} exists. If this is not the case, we choose  a $\lambda_{\min}$ value such that $d_F$ as in \eqref{indDf} is close to $ \inf_{\lambda >0} \frac{ d_{\rm SES}(\lambda)}{\gamma_\lambda} >0$.}
 
\end{remark}

In the literature, e.g. \cite{Kinnunen}, a property as in \eqref{eq:GlobFatness} is called a uniform (exterior) fatness property of the corresponding domain, and this notion is related to that of variational $2$-capacity, hence the name that we have chosen. 
We continue with a short discussion of this global indicator. 
\begin{description}
[leftmargin=20pt,style=nextline,itemsep=5pt,font=\mdseries,topsep=3pt]
\item[1)] 
It is clear from its definition that $d_F$ is  a global indicator and that it has the following upper and lower bounds
\begin{align*}
	d_F
	&\ge
	8 d_{\Omega^M} = 8 \max_{x \in\Omega^M} {\rm dist}(x, \partial \Omega^M), \\
	d_F &\le\lim_{\lambda\to \infty}   \frac{d_{\rm SES}(\lambda)}{\gamma_\lambda}  = 
	8\max_{x\in\Omega^M}
	{\rm dist}\left(x,\partial\mbox{conv}(\Omega^M)\right). 
\end{align*}
The quantity $d_F$ can be seen as a measure for the globularity of the domain $\Omega^M$ that involves the maximal distance to a ${\rm SES}$  ($d_{\rm SES}(\lambda_{\min})$) and the maximal distance  to the boundary  $\partial \Omega^M$ ($d_{\Omega^M}$).
\item[2)]
In the particular case where $\Omega^M$ consists of a linear chain of overlapping uniform balls, with radius $R$, of length $M$, we have $d_{\Omega^M}=R$, $d_{\rm SES}(\lambda)=R$, and thus $d_F= 8 R$, i.e., $d_F$ is independent of $M$. 
On the other hand, when considering a geometry of uniform balls whose centers lie on the unit grid $[1,L]^3$ with radius $=R>\sqrt{3}$ (in order that no inner holes appear), we obtain \rev{$d_{\Omega^M}=d_{\rm SES}(\lambda)=\frac12(L-1)+R$, hence, $d_F= 4(L-1)+8R$.} In this case $d_F$ is proportional to $L=M^{1/3}$ as $L$ increases.
\item[3)] 
Another consequence of the construction above  is that the entire cavity $\Omega^M$ can be covered by balls $B(y;R_y)$ with centers  $y$ on the ${\rm SES}_{\lambda_{\min}} ={\rm SES}(\lambda_{\min} d_{\Omega^M}) \subset \Omega^{M,c}$, radii $R_y \leq d_{\rm SES}(\lambda_{\min})$ and each of these balls contains a smaller ball of radius $\min \{\lambda_{\min} d_{\Omega^M},\frac12 R_y\}$ that lies entirely in $\Omega^{M,c}$.
\end{description}

\begin{figure}[t!]
	\centering
	\includegraphics[trim=0pt 140pt 0pt 140pt, clip, scale=0.5]{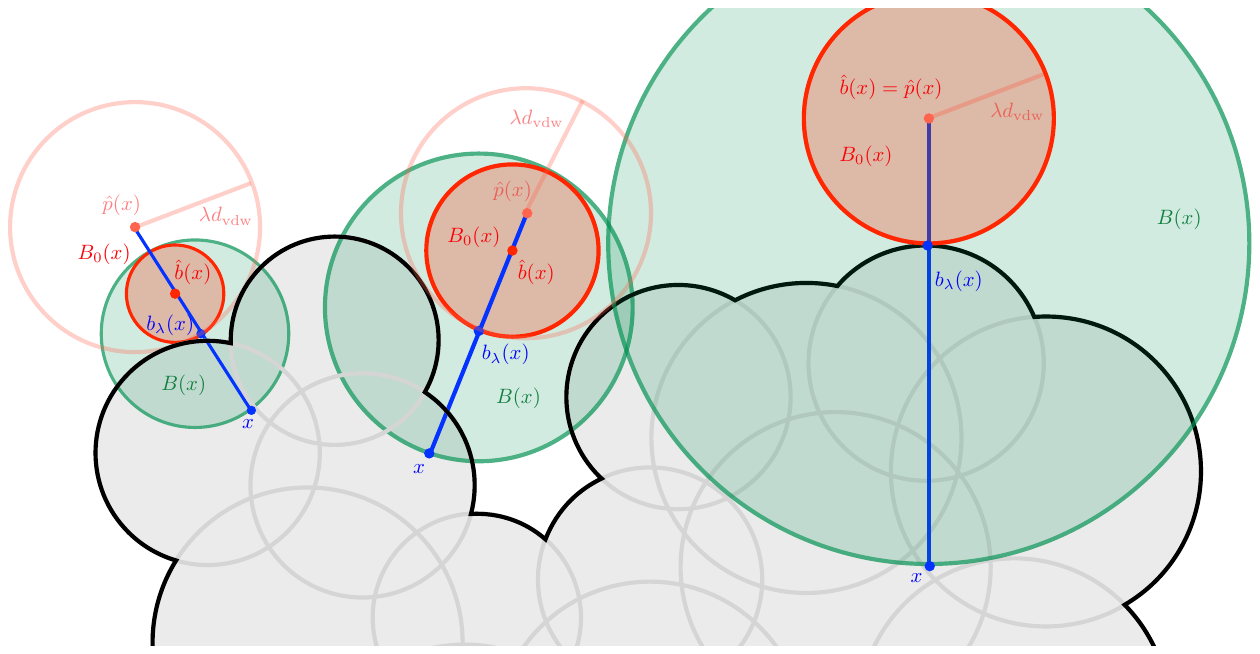}
\caption{Illustration of the geometrical setup to define the global fatness property.}
	\label{fig:ExtFat}
\end{figure}

\subsection{Partition of unity} \label{sectPU}
In this section we introduce a (natural) partition of unity, based on the local distance functions $\delta_i$, $i \in \cI$, and derive smoothness properties for the elements in this partion of unity. In the bounds for derivatives that are proven below only local geometry indicators from $\cG_L^M$ are involved.

%
%

Note that since $\delta_i$ is the distance function to the boundary in $\Omega_i$ we have
\begin{equation} \label{eq1}
  \|\nabla \delta_i(x)\| =1, ~~x \in \Omega_i, ~~\|\nabla \delta_i(x)\|=0, ~~x \notin \bar \Omega_i,
\end{equation}
One easily checks that $\delta(x)= \sum_{i=1}^M \delta_i =0$ if and only if $x \in \partial \Omega^M$. 
We further define
\[
	\theta_i(x):= \frac{\delta_i(x)}{\delta(x)}, \quad x \in \Omega^M.
\]
The system $(\theta_i)_{1 \leq i \leq M}$ forms a partition of unity (PU) of $\Omega^M$ subordinate to the cover $(\Omega_i)_{1\leq i\leq M}$:
\begin{align*}
   0 \leq \theta_i \leq 1 &~~(1 \leq i \leq M) \quad \text{on}~~\Omega^M,\\
  \sum_{i=1}^M \theta_i=1 & \quad \text{on}~~\Omega^M,\\
   {\rm supp}\, \theta_i \subset \bar \Omega_i & ~~(1 \leq i \leq M).
\end{align*}

\begin{remark} \label{remsmoothness} \rm 
 The functions $\theta_i$ are in general not smooth.  In  Figure~\ref{fig:Geom} a two-dimensional case is illustrated, in which $\Omega^M$ consists  of three intersecting disks  (in this case we have $\cI_{\intset}=\emptyset$). As a further, more precise, illustration we consider the three-dimensional case of two overlapping balls  $\Omega_1=B\big((-1,0,0);2\big)$, $\Omega_2=B\big((1,0,0);2\big)$, with $\Omega^2:=\Omega_1 \cup \Omega_2$. We thus have $\theta_i=\frac{\delta_i}{\delta}$, $i=1,2$, $\delta=\delta_1+\delta_2$. 
 
The intersection of $\partial \Omega^2$ with $\overline{\Omega_1 \cap \Omega_2}$ is  the circle $S=\left\{ (0,x_2,x_3)\; \middle| \; x_2^2+x_3^2=3\right\}$. The functions $\theta_i$ do \emph{not} have a 
 continuous  extension to the intersection circle $S$. Take an accumulation point $x \in S$  and a sequence $(x_n)_{n \in \Bbb{N}} \subset \Omega_1 \setminus \bar \Omega_2$ with $\lim_{n \to \infty} x_n=x$. We then have $\lim_{n\to \infty} \theta_1(x_n)=1$.  On the other hand,  we can take a sequence $(x_n)_{n \in \Bbb{N}} \subset (\Omega_1 \cap \Omega_2)$ with $\lim_{n \to \infty} x_n=x$ and $\delta_1(x_n) \leq \delta_2(x_n)$, which  implies $\limsup_{n\to \infty} \theta_1(x_n) \leq \frac12$. \black
Furthermore, elementary calculation yields that $\theta_i \notin H^1(\Omega_i)$. The key steps for the derivation of this result are outlined in  Appendix~\ref{App2}.

\rev{Related to this we note that the PU introduced above differs from the ones that are typically used in the analysis of domain decomposition methods, cf. \cite[Section 3.2]{Toselli}. The reason is that the overlapping covering $(\Omega_i)_{1 \leq i \leq M}$ does not satisfy the key assumption, cf. \cite[Assumption 3.1]{Toselli}, that for arbitrary $i$ and $x \in \Omega_i$, there exists $\tilde\delta_i >0$ and a $j$ such that ${\rm dist}(x, \partial \Omega_j \setminus \partial \Omega^M) \geq \tilde\delta_i$ holds.}
\end{remark}
%

\begin{lemma} \label{lemPU}
 The PU has the following properties:
 \begin{align} 
  \|\nabla\theta_i\|_{\infty,\Omega_i} 
  & \leq 
  \frac{\Nzero}{\gint}, \quad  \text{for all}~~i \in \cI_{\intset}, \label{PU2} 
  \\
  \|\nabla\theta_i(x)\| 
  & \leq
   \frac{\Nzero}{\gb \,{\rm dist}(x,\partial \Omega^M)},~~\text{for all}~~i \in \cI_{\bset},~x \in \Omega_i. \label{PU3}
 \end{align}
\end{lemma}
\begin{proof}
 Note that
 \[ \begin{split}
  \nabla \theta_i(x) & = \frac{\nabla \delta_i(x)}{\delta(x)} -  \frac{ \delta_i(x)}{\delta(x)}\cdot 
 \frac{\sum_{j=1}^M \nabla \delta_j(x)}{\delta(x)}, \\ 
\blue     & = \frac{\nabla \delta_i(x)}{\delta(x)} \left( 1-   \frac{ \delta_i(x)}{\delta(x)}\right) -  \frac{ \delta_i(x)}{\delta(x)}\cdot 
 \frac{\sum_{j\neq i} \nabla \delta_j(x)}{\delta(x)}\quad \text{a.e. on}~~\Omega^M. 
\end{split} \] 
\black Hence, cf. \eqref{eq1} and \eqref{card},
\begin{equation} \label{eq2}
  \|\nabla \theta_i(x)\|
   \leq 
   \frac{1}{\delta(x)}
   +
 \frac{\sum_{j\neq i} \|\nabla \delta_j(x)\|}{\delta(x)} \leq \frac{\Nzero}{\delta(x)}, \quad \text{a.e. on}~~\Omega^M. 
\end{equation}
For $i \in \cI_{\intset}$ we use \eqref{overlap1} and obtain the result \eqref{PU2}. 
For $i \in \cI_{\bset}$ we use \eqref{overlap3}, which yields the result \eqref{PU3}.
 \end{proof}
\ \\[1ex]
The results in \eqref{PU2}, \eqref{PU3} show that away from the boundary (i.e., in the subdomain $\Omega_{\intset}$) the partition of unity functions $\theta_i$ are in $W^{1,\infty}$, and towards the boundary the singularity of the gradient of these functions can be controlled by the distance function to the boundary.

\begin{figure}[t!]
	\centering
	\includegraphics[trim=60pt 40pt 70pt 100pt, clip, scale=0.25]{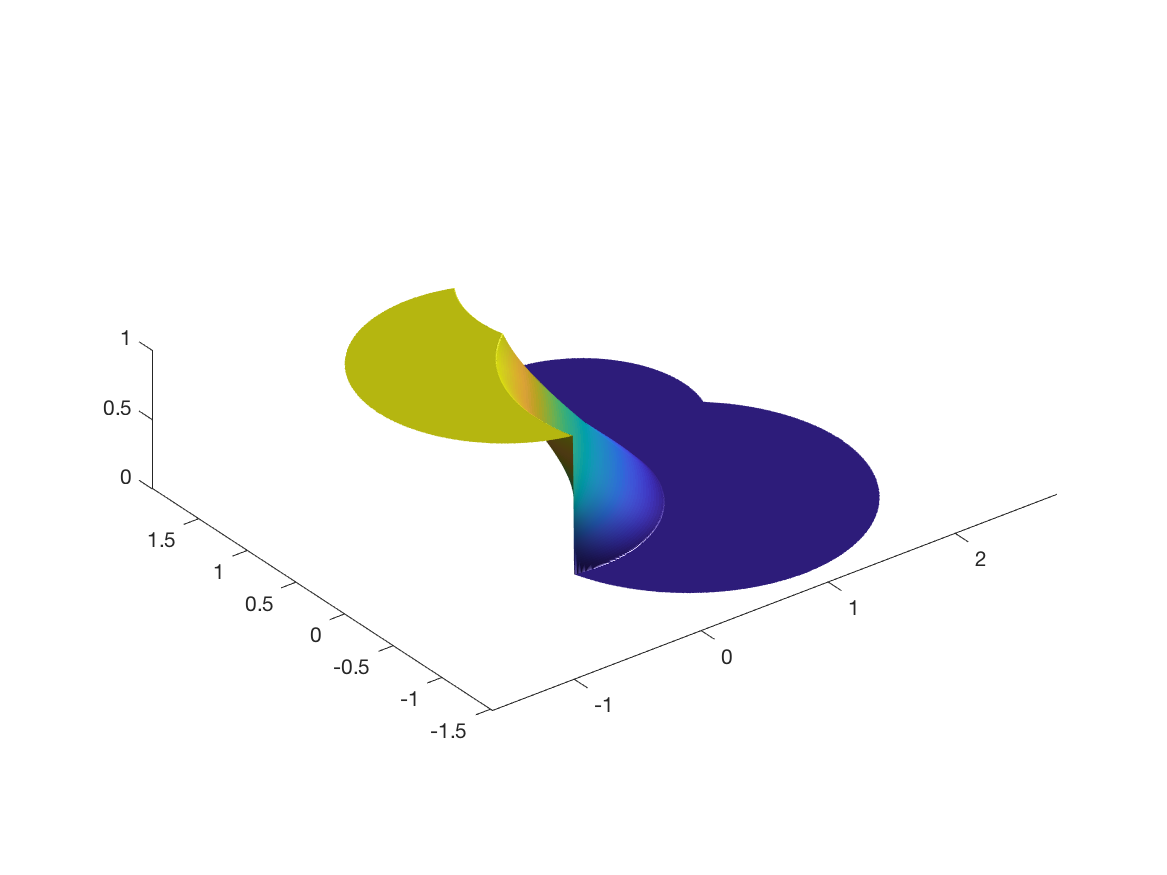}
	\includegraphics[trim=60pt 40pt 70pt 100pt, clip, scale=0.25]{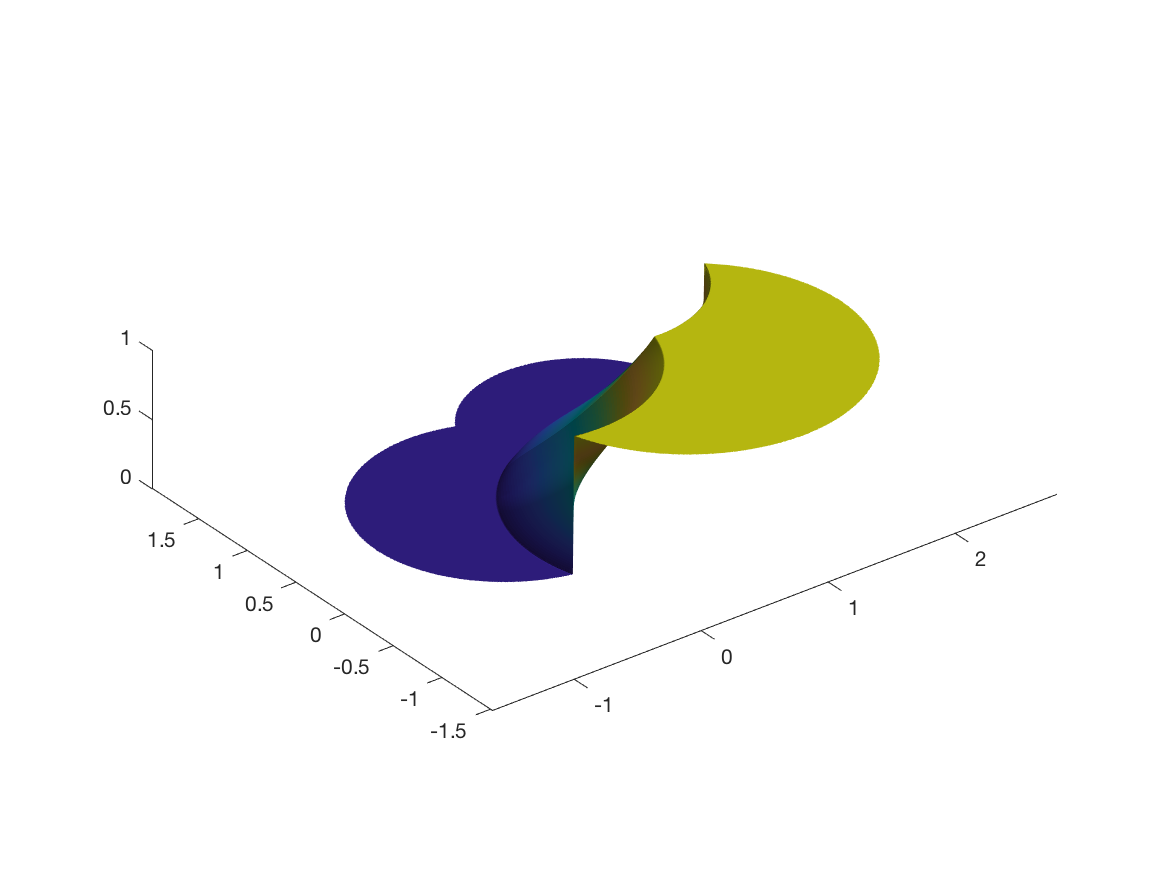}
	\includegraphics[trim=60pt 40pt 70pt 100pt, clip, scale=0.25]{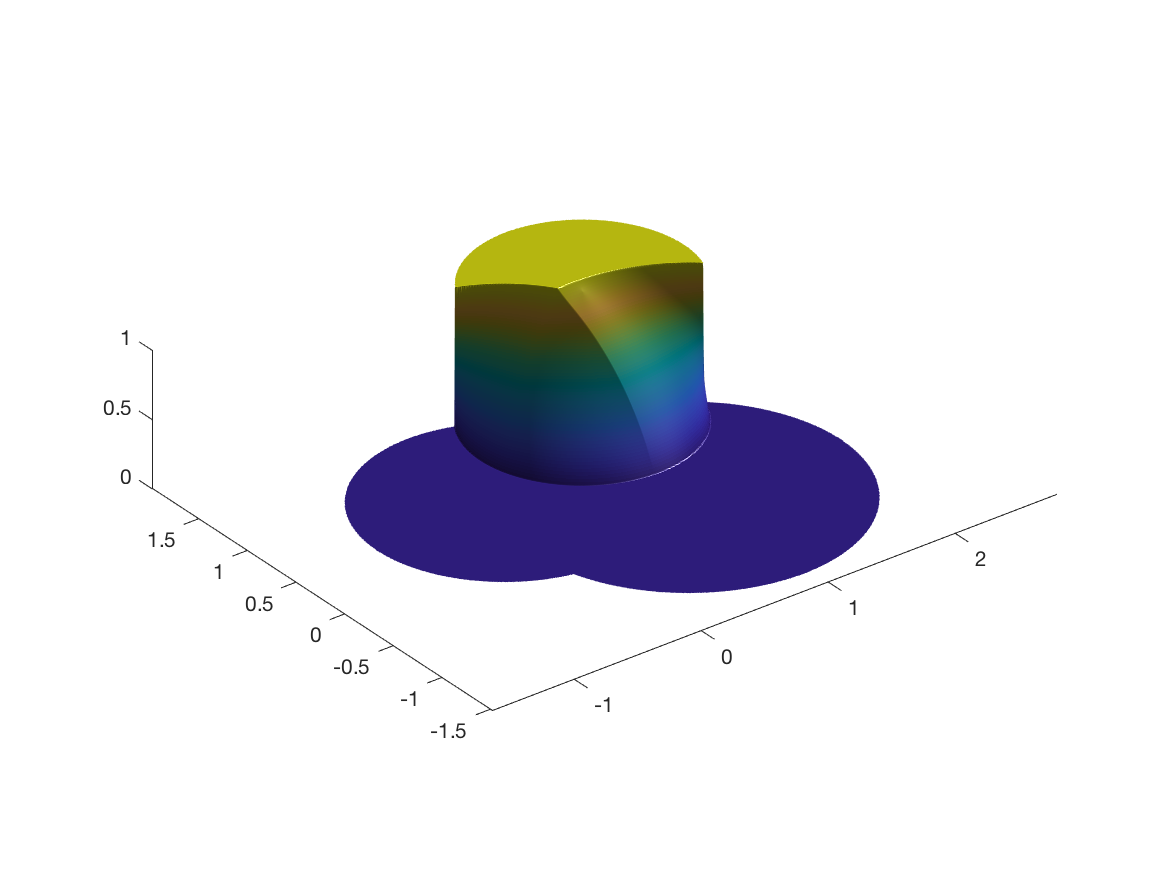}
\caption{Surface plot of $\theta_i$ for a configuration with 3 disks in two dimensions.}
	\label{fig:Geom}
\end{figure}

\section{Analysis of the Schwarz domain decomposition method} \label{SectAnalysis}
\subsection{Pointwise Hardy inequality} \label{sectHardy}
In the analysis we need a specific pointwise Hardy inequality, similar to the one derived in e.g. \cite{Hajlasz,Kinnunen}. We introduce some notation. For $f \in L^1(\Bbb{R}^3)$, $B \subset \Bbb{R}^3$ we denote the average by $f_B:= \frac{1}{|B|}\int_B f(y) \, dy$, and the \emph{maximal function} (cf. \cite{Stein}) by
\[
  \cM(f)(x):= \sup_{r >0} \frac{1}{|B(x;r)|} \int_{B(x;r)} |f(y)| \, dy.
\]
A key property is the following (\cite{Stein}, Theorem 1): for all $f \in L^2(\Bbb{R}^3)$:
\begin{equation} \label{estimatemaximal}
 \|\cM f\|_{L^2(\Bbb{R}^3)} \leq c_\cM \, \|f\|_{L^2(\Bbb{R}^3)}, \quad \text{with}~~c_\cM:=10 \sqrt{10}.
\end{equation}
In particular, if $f$ is only supported in  $\Omega^M$, then the right hand side in \eqref{estimatemaximal} reduces to $\|f\|_{L^2(\Omega^M)}$.

\begin{lemma} \label{LemHardy}
Consider a point $x \in \Omega^M$ and a corresponding (arbitrary) {\rm exterior} point denoted by $b(x) \in \Omega^{M,c}$. Define $d(x):=\|x-b(x)\|$. Let $\gamma(x) >0$ be such that
 \begin{equation} \label{condHardy}
   |B(b(x);d(x)) \cap \Omega^{M,c}| \geq \gamma(x) \, |B(b(x);d(x))|.
 \end{equation}
There exists a constant $c_\cH>0$, independent of $x$ and any parameter, such that the following holds for all $u \in C_0^\infty(\Omega^M)$:
\begin{equation} \label{Hardy1}
 |u(x)| \leq c_\cH\frac{d(x)}{\gamma(x)}  \cM(\|\nabla u\|)(x).
\end{equation}
\end{lemma}

%

\begin{proof} 
Take $u \in C_0^\infty(\Omega^M)$, extended by zero outside $\Omega^M$. 
Given  $x \in \Omega^M$ and a (fixed) $b(x)\in \Omega^{M,c}$, denote $B:=B(b(x);d(x))$, and $u_B$ the average of $u$ over $B$. 
Further, for this choice of $x$, take arbitrary $y \in B\cap \Omega^{M,c}$. 
Using Lemma 7.16  from \cite{GT} we obtain
\begin{equation} \label{Eq2}
 \begin{split}
  |u(x)| & = |u(x)-u(y)| \leq |u(x)-u_B| +|u(y)-u_B| \\
  & \leq \ssc_1\left( \int_B \frac{\|\nabla u(z)\|}{\|x-z\|^2} \, dz + \int_B \frac{\|\nabla u(z)\|}{\|y-z\|^2} \, dz \right),
 \end{split}
\end{equation}
with $\ssc_1= \frac{2}{\pi}$. We recall the elementary inequality (e.g., Lemma 3.11.3 in \cite{Ziemer}):
\[
 \frac{1}{|B(x_0;r)|} \int_{B(x_0;r)} \frac{1}{\|y-z\|^2} \, dy \leq \ssc_2 \frac{1}{\|x_0-z\|^2} \quad \text{for all}~~r >0, ~x_0,z \in \Bbb{R}^3.
\]
(Inspection of the proof of Lemma 3.11.3 in \cite{Ziemer} yields $\ssc_2 \leq 14$).
Using this, the definition of $\gamma(x)$ and that $B\cap \Omega^{M,c} \subset B  \subset B(x;2d(x))$ we obtain:
\begin{align*}
 \inf_{y \in B\cap \Omega^{M,c}} \int_B \frac{\|\nabla u(z)\|}{\|y-z\|^2} \, dz  \leq & \frac{1}{|B\cap \Omega^{M,c}|} \int_{B\cap \Omega^{M,c}} \int_B  \frac{\|\nabla u(z)\|}{\|y-z\|^2} \, dz \, dy \\
 \leq & \frac{1}{\gamma(x)} \int_B \frac{1}{|B|}\int_{B \cap \Omega^{M,c}} \frac{1}{\|y-z\|^2} \, dy \, \|\nabla u(z)\| \, dz \\
 \leq & \frac{8}{\gamma(x)} \int_B \frac{1}{|B(x;2d(x))|}\int_{B(x;2d(x)) } \frac{1}{\|y-z\|^2} \, dy \, \|\nabla u(z)\| \, dz \\
 \leq &  \frac{8 \, \ssc_2 }{\gamma(x)} \int_B\frac{\|\nabla u(z)\|}{\|x-z\|^2} \, dz.
\end{align*}
Using this in \eqref{Eq2} and $\gamma(x) \leq 1$, we get
\[
 |u(x)| \leq \ssc_1\left(1+\frac{8 \, \ssc_2 }{\gamma(x)}\right)\int_B \frac{\|\nabla u(z)\|}{\|x-z\|^2} \, dz \leq \frac{ \ssc_1(1+8 \ssc_2)}{\gamma(x)} \int_{B(x;2d(x))} \frac{\|\nabla u(z)\|}{\|x-z\|^2} \, dz.
\]
Finally we use the following estimate (\cite{Hedberg}, Lemma 1):
\[
 \int_{B(x;r)} \frac{|f(z)|}{\|x-z\|^2}\, dz \leq \ssc_3 \, r \, \cM(f)(x) \quad \text{for all}~r >0.
\]
(Inspection of the proof in \cite{Hedberg} yields $\ssc_3 \leq \frac{4}{\ln 2}$).
Combining these results we obtain the estimate  \eqref{Hardy1} with $c_\cH:=2\,\ssc_1\,(1+8 \, \ssc_2)\,\ssc_3$.
\end{proof}
\\

The result of this  lemma is essentially the same as in Proposition 1 in \cite{Hajlasz} and in Theorem 3.9 in \cite{Kinnunen}. The only difference is that in the latter results a specific choice for  the point $b(x)$ is used, namely a closest point projection of $x$ onto the boundary. Hence, in that case one has $d(x)=\|x-b(x)\|={\rm dist}(x, \partial \Omega^M)$. We will use this specific choice also in the proof of 
Corollary~\ref{Coro1} below. In section~\ref{SectPoincare}, however, we will use a different choice for $b(x)$.

\begin{corollary} \label{Coro1}
The following holds (recall that $\Omega_{\bset} =\cup_{i \in \cI_{\bset}} \Omega_i$):
\[
  \left(\int_{\Omega_{\bset}} \left(\frac{u(x)}{ {\rm dist}(x,  \partial \Omega^M)}\right)^2\, dx\right)^\frac12 \leq \frac{c_\cM c_\cH}{\cf} \, |u|_{1,\Omega^M} \quad \text{for all}~u \in H^1_0(\Omega^M),
\]
with $c_\cM$ as in~\eqref{estimatemaximal}, $c_\cH$ as in~\eqref{Hardy1} and $\cf$ from \eqref{fatness}.
\end{corollary}
\begin{proof}
Due to density it suffices to consider $u \in C_0^\infty(\Omega^M)$. Take $x \in \Omega_{\bset}$ and $i\in \cI_{\bset}$ such that $x \in \Omega_i$. We take for $b(x)\in  \Omega^{M,c}$ the closest point projection on $ \Gamma_i$ as in Indicator~\ref{indic45}, hence $d(x)=\|x-b(x)\| = {\rm dist}(x,  \Gamma_i ) \leq  {\rm dist}(x,  \partial \Omega^M)$.
For $b(x)$ the fatness estimate~\eqref{fatness1} holds. 
Combining this with the pointwise Hardy estimate~\eqref{Hardy1} and~\eqref{estimatemaximal} yields
\begin{align*}
	 &\left(\int_{\Omega_{\bset}} \left(\frac{u(x)}{ {\rm dist}(x,  \partial \Omega^M)}\right)^2\, dx\right)^\frac12
	\leq \left( \int_{\Omega_{\bset}}  \frac{c_\cH^2 }{\cf^2}  \cM(\|\nabla u\|)(x)^2 \, dx \right)^\frac12 
	\\ & 
	\leq  \frac{c_\cH  }{\cf}    \|\cM(\|\nabla u\|)\|_{L^2(\Bbb{R}^3)} 
	\leq c_\cM \frac{c_\cH}{\cf}   \|\nabla u\|_{L^2(\Bbb{R}^3)} 
	= \frac{c_\cM c_\cH }{\cf}|u|_{1,\Omega^M}, 
\end{align*}
which completes the proof.
\end{proof}

\subsection{Stability of the subspace decomposition}
In the definition of the stability constant $s_0$ in \eqref{c0} one is free to choose for $v \in H_0^1(\Omega^M)$ any decomposition $v = \sum_{i=1}^M v_i$ with $v_i \in H_0^1(\Omega_i)$. Below we use the natural choice $v_i=\theta_i v$ and analyze this particular decomposition.
The result in the following theorem is crucial for the analysis of the Schwarz method. 

\begin{theorem} \label{thm1}
 The following estimates hold for all $v \in H_0^1(\Omega^M)$:
\begin{equation} 
 	\sum_{i=1}^M |\theta_i v|_{1,\Omega_i}^2  
 	\leq {\sC}_1 |v|_{1,\Omega^M}^2 + {\sC}_2 \sum_{i \in \cI_{\intset}} \|v\|_{L^2(\Omega_i)}^2 
 	\leq {\sC}_1|v|_{1,\Omega^M}^2 + {\sC}_2 \, \Nzero \|v\|_{L^2(\Omega^M)}^2,  \label{resthm1}
\end{equation}
where the constants ${\sC}_1,{\sC}_2$ depend only on the local geometry indicators in $\cG_L^M$; in particular  ${\sC}_2=2\Nzero^2/\gint^2$.
\end{theorem}
\begin{proof}
The second inequality in \eqref{resthm1} is an easy consequence of the definition of the overlap indicator $\Nzero$. 
We derive the first inequality. 
Take $v \in H_0^1(\Omega^M)$. Using $0 \leq \theta_i \leq 1$ and the definition of $\Nzero$ we get
\begin{align}
 \sum_{i=1}^M |\theta_i v|_{1,\Omega_i}^2 & = \sum_{i=1}^M \int_{\Omega_i} \|\nabla (\theta_i v)\|^2 \,dx  \nonumber 
    \leq 2 \sum_{i=1}^M\int_{\Omega_i} \theta_i^2 \|\nabla v\|^2 \,dx +  2 \sum_{i=1}^M\int_{\Omega_i} \|\nabla \theta_i\|^2 v^2 \, dx \nonumber 
    \\ 
    &  \leq 2 N_0 |v|_{1,\Omega^M}^2 + 2 \sum_{i=1}^M\int_{\Omega_i} \|\nabla \theta_i\|^2 v^2 \, dx. \label{heq1}
\end{align}
For estimating the second term in \eqref{heq1} we use the partitioning $\{1,\ldots,M\}=\cI_{\intset} \cup \cI_{\bset}$. Using the result \eqref{PU2} we obtain
\begin{equation} \label{heq2} 
     \sum_{i \in \cI_{\intset}} \int_{\Omega_i} \|\nabla \theta_i\|^2 v^2 \, dx  
     \leq \left( \frac{\Nzero}{\gint}\right)^2  \sum_{i \in \cI_{\intset}} \|v\|_{L^2(\Omega_i)}^2.
 \end{equation}
We finally analyze the $\sum_{i \in \cI_{\bset}}$ part of the sum in \eqref{heq1}.  Take $ i \in \cI_{\bset}$, $x \in \Omega_i$. 
Using \eqref{PU3} and Corollary~\ref{Coro1} we obtain
 \begin{align*}
  \sum_{i \in \cI_{\bset}} \int_{\Omega_i}  \|\nabla \theta_i\|^2 v^2 \, dx  
  & \leq \left(\frac{\Nzero}{\gb}\right)^2 \sum_{i \in \cI_{\bset}} \int_{\Omega_i} \left( \frac{v}{{\rm dist}(x,\partial \Omega^M)}\right)^2 \, dx  \\ 
  & \leq \frac{\Nzero^3}{\gb^2} \,  \int_{\Omega_{\bset}}\left( \frac{v}{{\rm dist}(x,\partial \Omega^M)}\right)^2 \, dx 
 \leq \frac{\Nzero^3 c_\cM^2 c_\cH^2}{\gb^2 \cf^2} |u|_{1,\Omega^M}^2.
 \end{align*}
%
Combining this with  \eqref{heq1}, \eqref{heq2} completes the proof, yielding
\[	
	{\sC}_1= 2\Nzero\left(1 +  \frac{\Nzero^2 c_\cM^2 c_\cH^2}{\gb^2 \cf^2}\right).
\]
\end{proof}
\ \\
 The result \eqref{resthm1}  implies that, although $\theta_i$, $i \in \cI_\bset$, is \emph{not} necessarily in $H^1(\Omega_i)$ (Remark~\ref{remsmoothness}), for $v \in H_0^1(\Omega^M)$ the product $\theta_i v$ is an element of $H^1(\Omega_i)$. In this product the singularity of $\nabla \theta_i$ at the boundary can be controlled due to the property $v_{|\partial \Omega^M}=0$.  
\begin{corollary}\label{corolA}  Assume $\cI_{\intset}=\emptyset$. The estimate 
 \begin{equation} 
 	\sum_{i=1}^M |\theta_i v|_{1,\Omega_i}^2  
 	\leq {\sC}_1 |v|_{1,\Omega^M}^2, \quad \text{for all}~ v \in H_0^1(\Omega^M),
\end{equation}
holds, where the constant ${\sC}_1$ depends only on the local geometry indicators in $\cG_L^M$.
\end{corollary}
 
\begin{corollary} For all $v \in H_0^1(\Omega^M)$ we have
\begin{equation} \label{res1cor}
 \theta_i v \in H_0^1(\Omega_i), ~~i =1, \ldots, M.
\end{equation}
Furthermore,
\begin{equation} \label{res2cor}
 \sum_{i=1}^M H_0^1(\Omega_i)= H_0^1(\Omega^M)
\end{equation}
holds. 
\end{corollary}
\begin{proof}
Take $v \in H_0^1(\Omega^M)$ and $i \in \cI$. From Theorem~\ref{thm1} it follows that $\theta_i v \in H^1(\Omega_i)$. Define $\Gamma_i:=\partial \Omega_i \cap \partial \Omega^M$, i.e., $\partial \Omega_i = \Gamma_i \cup (\partial \Omega_i \setminus \Gamma_i)$. If ${\rm meas}_2(\Gamma_i) >0 $, then on $\Gamma_i$ we have (due to the trace theorem) $(\theta_i v)_{|\Gamma_i}= (\theta_i)_{|\Gamma_i} v_{|\Gamma_i}=0$, due to $v_{|\partial \Omega^M}=0$ and boundedness of $\theta_i$.  For $x \in   \partial \Omega_i \setminus \Gamma_i $ we have $\theta_i(x)=0$, hence, $(\theta_i v)_{| (\partial \Omega_i \setminus \Gamma_i)}=0$. This completes the proof of \eqref{res1cor}.

Take $v \in H_0^1(\Omega^M)$ and note that
$
  v = \left( \sum_{i=1}^M \theta_i\right) v = \sum_{i=1}^M (\theta_i v),
$
with $\theta_i v \in H_0^1(\Omega_i)$. This proves the result \eqref{res2cor}.
\end{proof}
\ \\

The result \eqref{res2cor} implies that the assumption used in Lemma~\ref{lemmaXu} is satisfied:
\begin{corollary}
 $\sum_{i=1}^M H_0^1(\Omega_i)$ is closed in $H_0^1(\Omega^M)$.
\end{corollary}
\ \\

The remaining task is to bound the term 
$\sum_{i \in \cI_{int}} \|v\|_{L^2(\Omega_i)}^2$ in  \eqref{resthm1}. 
 In section~\ref{SectPoincare} we derive a Poincar\'e estimate in which this term is bounded in terms of the desired norm $|v|_{1,\Omega^M}$. This immediately implies a bound for the stability constant $s_0$. 


\subsection{Poincar\'e estimate and main result}
\label{SectPoincare} 
\rev{In this section we derive a main result, namely an estimate for the stability quantity $s_0$ in Theorem~\ref{thmstab} below. We summarize the different parameters that have been introduced above and play a role in this main result. These are:
\begin{align*}
 & \text{local geometry indicators, cf. Section~\ref{Sectgeometry}}:~\cG_L^M:=\{\Nmax,\Nzero,\gint^{-1}, \gb^{-1},\cf^{-1}\},\\
 & \text{global exterior fatness indicator, cf. \eqref{indDf}}:~d_F,\\
 & \text{generic Hardy constants, cf. \eqref{estimatemaximal}, \eqref{Hardy1}}:~~c_\cM, ~c_\cH.
\end{align*}
}
We use the pointwise Hardy inequality in Lemma~\ref{LemHardy} in combination with the \emph{global} exterior fatness property as in indicator~\ref{ind46} to derive a bound for the term $\|v\|_{L^2(\Omega^M)}$ in terms of $|v|_{1,\Omega^M}$. 
Note that the optimal constant that occurs in such an estimate is the inverse of the smallest eigenvalue of the Laplace eigenvalue problem in $H_0^1(\Omega^M)$. 
In the following, we will find a geometry-dependent upper bound of this constant, or equivalently, a lower bound of the smallest eigenvalue, that  depends on geometric features of the domain $\Omega^M$. 
For this we use the global fatness indicator in \eqref{indDf} which characterizes certain global geometry properties of $\Omega^M$ (which then, for example,  distinguishes a linear chain from a globular topology).

\begin{lemma} \label{Lemma5} The following estimate holds for all $u \in H_0^1(\Omega^M)$:
 \[
  \|u\|_{L^2(\Omega^M)} \leq c_\cH c_\cM d_F |u|_{1,\Omega^M},
\]
with the global fatness indicator $d_F$ as in \eqref{indDf}.
\end{lemma}
\begin{proof} Due to density it suffices to consider $u \in C_0^\infty(\Omega^M)$. 
For $x \in \Omega^M$ we take the exterior point $b(x)$  and $\gamma_F$ as in  Indicator~\ref{ind46}. Using  Lemma~\ref{LemHardy} and \eqref{estimatemaximal}  we obtain with $d(x):=\|b(x)-x\|$:
\begin{align*}
 	 \|u\|_{L^2(\Omega^M)} 
  	& = 
  	\left( \int_{\Omega^M} u(x)^2 \, dx\right)^\frac12 
  	\leq 
  	\frac{c_\cH}{\gamma_F} \max_{x \in \Omega^M} d(x) 
	\left( \int_{\Omega^M}   \cM(\|\nabla u\|)(x)^2 \, dx \right)^\frac12 
	\\ & 
	\leq 
	c_\cH d_F    \|\cM(\|\nabla u\|)\|_{L^2(\Bbb{R}^3)}  
	\leq  
	c_\cH c_\cM d_F   \|\nabla u\|_{L^2(\Bbb{R}^3)} 
	=   
	c_\cH c_\cM d_F |u|_{1,\Omega^M}, 
\end{align*}
which proves the result.
\end{proof}

\begin{remark} \rm
It follows therefore that 
\[
	c_\cH^{-1} c_\cM^{-1} d_F^{-1}
\]
is a lower bound of the lowest eigenvalue of the Laplace eigenvalue problem in $H^1_0(\Omega^M)$, where $d_F$  accounts for the geometry of the domain.
\end{remark}
\ \\
\medskip
We derive a bound for the stability quantity $s_0$, cf. \eqref{lemmaXu}:
\begin{theorem} 
\label{thmstab}
 For the stability quantity $s_0$ the following bounds hold, with $\sC_3,\sC_4$, constants that depend only on the local geometry indicators in $\cG_L^M$ \rev{and on $c_\cH$, $c_\cM$.} If $\cI_{\intset}=\emptyset$ holds we have
 \begin{equation} \label{resmain1}
   s_0 \leq \sC_3.
 \end{equation}
If $\cI_{\intset} \neq \emptyset$  we have
\begin{equation} \label{resmain2}
   s_0 \leq \sC_3 + \sC_4 d_F^2.
 \end{equation}
\end{theorem}
\begin{proof}
For $v \in H_0^1(\Omega^M)$ we define $\hat v_j:= \theta_j v \in H_0^1(\Omega_j)$, i.e., $v = \sum_{j=1}^M \hat v_j$. Note that: 
\begin{align*}
 \inf_{\sum_{j=1}^M v_j =v}\sum_{i=1}^M  \Big|P_i \sum_{j=i+1}^M v_j \Big|_{1,\Omega_i}^2  
  &\leq \sum_{i=1}^M  \Big|P_i \sum_{j=i+1}^M  \hat v_j \Big|_{1,\Omega_i}^2
   \rev{=  \sum_{i=1}^M \Big|P_i \sum_{j\in \cN_i} \hat v_j  \Big|_{1,\Omega_i}^2}
   \\
  & \rev{\leq  \sum_{i=1}^M \Big|\sum_{j\in \cN_i} \hat v_j  \Big|_{1,\Omega^M}^2 \le \Nmax \sum_{i=1}^M \sum_{j\in \cN_i} |\hat v_j |_{1,\Omega_j}^2 } \\
& \rev{\le \Nmax^2 \sum_{i=1}^M | \hat v_i |_{1,\Omega_i}^2.}
\end{align*} 
\black
For estimating the term $\sum_{i=1}^M | \hat v_i |_{1,\Omega_i}^2$ we apply Theorem~\ref{thm1}. If $\cI_{\intset}= \emptyset$ we use Corollary~\ref{corolA} and then, using the definition of $s_0$, obtain the result in \eqref{resmain1}, with $\sC_3:=  \Nmax^2 {\sC}_1$ and ${\sC}_1$ as in Theorem~\ref{thm1}.  If $\cI_{\intset} \neq \emptyset$ we apply Theorem~\ref{thm1} and Lemma~\ref{Lemma5} and obtain the result \eqref{resmain2}, with $\sC_4= \rev{\Nmax^2}  {\sC}_2 \Nzero c_\cH^2 c_\cM^2$ (${\sC}_2$ as in Theorem~\ref{thm1}).
\end{proof}

\medskip
Note that the constants $\sC_3$, $\sC_4$ depend only on the local geometry indicators in $\cG_L^M$ (and the generic constants $c_\cH$, $c_\cM$). The information on the global geometry of $\Omega^M$ enters (only) through the global fatness indicator $d_F$.
This proves that in cases where one has a very large number of balls ($M \to \infty$) but moderate values of the geometry indicators in $\cG_L^M$  and of $d_F$, the convergence of the Schwarz domain decomposition method is expected to be fast. 
Furthermore, if  the number $M$ of balls is increased, but the values of these local geometry indicators and of $d_F$ are uniformly bounded with respect to $M$, the convergence of the  Schwarz  domain decomposition method does not deteriorate.  A worst case scenario (globular domain) is $d_F \sim M^\frac13$ ($M \to \infty$). In that case, due to $|E|_{1,\Omega^M}= (1- \frac{1}{1+s_0})^\frac12$, the number of iterations $\ell$ of the Schwarz domain decomposition method that is needed to obtain a given accuracy, scales like
\[ 
  \ell \sim s_0 \sim  d_F^2 \sim M^\frac23.
\]
\rev{A symmetrized version of the Schwarz domain decomposition method  can be combined with a conjugate  gradient acceleration, cf. \cite{Xuoverview}. This results in an improved scaling of the number of iterations, namely} $\ell \sim M^\frac13$.   

\section{An additional coarse global  space} \label{sectcoarse}
In case of a globular domain $\Omega^M$ both numerical experiments and the theory presented above show that the convergence of the Schwarz domain decomposition method can be  slow for very large $M$. 
As is well-known from the field of domain decomposition methods (and subspace correction methods)  this deterioration can be avoided by introducing a suitable ``coarse level space''. In this section we propose such a space:
\[
  V_0:={\rm span}\{ \, \theta_i~|~i \in \cI_{\intset}\,\} \subset H_0^1(\Omega^M)
\]
(we assume $\cI_{\intset}\neq \emptyset $, otherwise we use Theorem~\ref{thmstab}, Eqn. \eqref{resmain1}). The corresponding projection $P_0:\, H_0^1(\Omega) \to V_0$ is such that $a(P_0v,w_0)=a(v,w_0)$ for all $w_0 \in V_0$.
In the definition of $V_0$ it is important to restrict to $i \in \cI_{\intset}$ (instead of $i \in \cI$) because for $i \in \cI_{\bset}$ the partition of unity functions $\theta_i$ are not necessarily contained in $H_0^1(\Omega^M)$.
In the Schwarz method one then has to solve for an additional correction in $V_0$: 
 $e_0 \in V_0$ such that
\[ a(e_0,v_0)=f(v_0)-a(u_{i-1}^{\ell-1},v_0) \quad \text{for all}~~v_0 \in V_0.
 \]
Using the basis $(\theta_i)_{i \in \cI_{\intset}}$ in $V_0$ this results in   a sparse linear system of (maximal) dimension $\sim M \times M$. In practice this coarse global system will be solved approximately by a multilevel technique, cf. Remark~\ref{remapproximate} below. 

The analysis of the method with the additional correction in the space $V_0$ is again based on Lemma~\ref{lemmaXu}, which is also valid if we use the decomposition $H_0^1(\Omega)= V_0 +\sum_{i=1}^M H_0^1(\Omega_i)$ (and include $i=0$ in the sum in Lemma~\ref{lemmaXu}). 
For the analysis (only) we need a ``suitable'' linear mapping $Q_0: H^1_0(\Omega^M) \to V_0$. A natural choice is the following:
\[
  Q_0v= \sum_{i \in \cI_{\intset}} \bar v_i \theta_i, \qquad \bar v_i:= \frac{1}{|\Omega_i|} \int_{\Omega_i} v \, dx.
\]
For deriving properties of this mapping we introduce additional notation:
\begin{align*}
  \cN_i^\ast &
  := 
  \{\, j \in \cI_{\intset}~|~ \Omega_j \cap \Omega_i \neq \emptyset \, \} \subset \cN_i, \quad i \in \cI, 
  \\
  \cI_{\intset}^\ast
  &:=
   \{\, i\in  \cI_{\intset}~|~ [\, \Omega_j \cap \Omega_i \neq \emptyset\,] ~\Rightarrow ~j \in \cI_{\intset}\,\}, ~~\cI_{\bset}^\ast= \cI \setminus \cI_{\intset}^\ast,
  \\
  \Omega_i^\ast 
  &:= 
  \cup_{j \in \cN_i^\ast \cup \{i\}} \Omega_j,  \quad i \in \cI.
\end{align*}
The index  set $\cI_{\intset}^\ast$ contains those $i \in \cI_{\intset}$ for which the ball  $\Omega_i$ has only neighboring balls that are interior balls. For  $i \in \cI_{\intset}^\ast$ we have $\cN_i^\ast=\cN_i$.
We also need two further local geometry indicators:
\begin{align}
 	\Nzero^\ast~\text{is the smallest integer such that:}&~\max_{x\in \Omega^M}{\rm card}\{\, j~|~x\in \Omega_j^\ast\,\} \leq \Nzero^\ast. \label{cardA} \\
    q_i:=\max_{j \in \cN_i^\ast} \frac{|\Omega_i|}{|\Omega_j|}, &\quad q_{\max} :=\max_{i \in \cI} q_i.
\end{align}
The indicator in \eqref{cardA}  is a variant of the maximal overlap indicator $N_0$ in \eqref{card} (with $\Omega_j$ replaced by $\Omega_j^\ast$) \rev{and $q_{\max}$ essentially measures the maximal variation in radii of neighboring balls.} Below we use the standard Sobolev norm on $H^1(\Omega_i)$, denoted by $\|\cdot\|_{1,\Omega_i}$.
\begin{lemma} There are constants ${\sf B}_1,{\sf B}_2$,  depending only on the local geometry indicators $\Nzero, \Nzero^\ast,\gint, q_{\max}, R_{\max}, \Nmax$ such that:
\begin{align}
  \|v - Q_0v\|_{1,\Omega_i} &\leq |v|_{1,\Omega_i} + {\sf B}_1 \|v\|_{L^2(\Omega_i^\ast)}\quad \text{for all}~~i\in \cI,~~v \in H^1(\Omega_i^\ast), \label{RR1} \\
 \|v - Q_0v\|_{1,\Omega_i} & \leq  {\sf B}_2 |v|_{1,\Omega_i^\ast}\quad \text{for all}~~i\in \cI_{\intset}^\ast,~~v \in H^1(\Omega_i^\ast). \label{RR2}
\end{align}
\end{lemma}
\begin{proof}
Take $i \in \cI$, $v \in H^1(\Omega_i^\ast)$.  Note $(Q_0v)_{|\Omega_i} = \big( \sum_{j \in \cN_i^\ast} \bar v_j \theta_j\big)_{|\Omega_i}$. Since $\cN_i^\ast \subset \cI_{\intset}$ we have $\|\nabla \theta_j\|_{\infty,\Omega_i} \leq \frac{\Nzero}{\gint}$ for $j \in \cN_i^\ast$, hence, $\|  \theta_j\|_{1,\Omega_i}^2 \leq |\Omega_i|\left(1+\big(\frac{\Nzero}{\gint}\big)^2\right)$. Furthermore, $|\bar v_j| \leq |\Omega_j|^{-\frac12} \|v\|_{L^2(\Omega_j)}$. Using this we obtain
\begin{align*}
 	\|v - Q_0v\|_{1,\Omega_i} 
 	& \leq 
 	\|v\|_{1,\Omega_i} + \Big\|\sum_{j\in \cN_i^\ast}  \bar v_j \theta_j\Big \|_{1,\Omega_i}  
 	 \leq  
 	|v|_{1,\Omega_i} + \|v\|_{L^2(\Omega_i)} + \sum_{j\in \cN_i^\ast}  |\bar v_j| \|\theta_j \|_{1,\Omega_i} 
 	\\
 	& \leq  
 	|v|_{1,\Omega_i} + \|v\|_{L^2(\Omega_i)}  + \left(1+\frac{\Nzero}{\gint} \right)\max_{j \in \cN_i^\ast} \left(\frac{|\Omega_i|}{|\Omega_j|}\right)^\frac12 \sum_{j \in \cN_i^\ast} \|v\|_{L^2(\Omega_j)}
 	\\
	& \leq 
	|v|_{1,\Omega_i} + {\sf B}_1 \|v\|_{L^2(\Omega_i^\ast)},
\end{align*}
with ${\sf B}_1=1+\left(1+\frac{\Nzero}{\gint}\right) q_{\max}^\frac12 \Nzero^\ast$, which proves the result \eqref{RR1}. 

Now take $i \in \cI_{\intset}^\ast$, hence $(\sum_{j \in \cN_i^\ast} \theta_j)_{|\Omega_i}=1$. This implies that for an arbitrary constant $\hat c$ we have
\[
   (Q_0 \hat c)_{|\Omega_i}= \Big(\sum_{j \in  \cN_i^\ast} \hat c \, \theta_j\Big)_{|\Omega_i} = \hat c. 
\]
Using the estimate from above yields, for arbitrary $v \in H^1(\Omega_i^\ast)$
\[
	 \|v - Q_0v\|_{1,\Omega_i}
 	=
 	\|v-\hat c - Q_0(v-\hat c)\|_{1,\Omega_i} \leq |v|_{1,\Omega_i} + {\sf B}_1   \|v-\hat c\|_{L^2(\Omega_i^\ast)}
\]
for an arbitrary constant $\hat c$. Take $\hat c := \frac{1}{|\Omega_i^\ast|} \int_{\Omega_i^\ast} v \, dx$, hence $\int_{\Omega_i^\ast} v-\hat c \, dx =0$. We now apply a Friedrichs inequality to the term $\|v-\hat c\|_{L^2(\Omega_i^\ast)}$. The domain $\Omega_i^\ast$ has a simple structure, namely the union of a few (at most $\Nmax$) balls. Hence the constant in the Friedrichs inequality can be quantified.  Theorem 3.2 in \cite{Zheng} yields an estimate in which the constant $c_F$ depends only on the number $\Nmax$ and the diameters $R_j$, $j \in \cN_i^\ast$. This yields
\[ \|v-\hat c\|_{L^2(\Omega_i^\ast)} \leq c_F(\Nmax,R_{\max}) |v|_{1,\Omega_i^\ast}.
\]
Hence, \eqref{RR2} holds, with ${\sf B}_2=1+{\sf B}_1 c_F(\Nmax,R_{\max})$.
\end{proof}
\ \\[1ex]
In the derivation of the main result in Theorem~\ref{thmstab2} below we have to control $\|v\|_{L^2(\Omega_b^\ast)}$, on a ``boundary strip'' $\Omega_b^\ast:=\cup_{i \in \cI_b^\ast} \Omega_i^\ast$, in terms of $|v|_{1,\Omega^M}$. For this we use arguments, based on the Hardy inequality, very similar to the ones used in the previous sections. Denote the union of all balls that have a nonzero overlap with $\Omega_i$ by $\Omega_i^e:=\cup_{j \in \cN_i} \Omega_j$. Note that for $i \in \cI_b^\ast$ the ball $\Omega_i$ has an overlapping neighboring ball $\Omega_j \subset \Omega_b$, and thus $\partial \Oie \cap \partial \Omega^M=:\Gie \neq \emptyset$. We need a variant of the local exterior fatness Indicator~\ref{indic45}.
For $ i \in \cI_b^\ast$ and $x \in \Omega_i^\ast$ let   $p(x)$ be the closest point projection on $\Gie$.  Let   $\hat c_i^\ast >0$ be such that
\begin{equation} \label{fatness1A}
	 |B\big(p(x); \|p(x)-x\|\big)\cap \Omega^{M,c}| \geq \hat c_i^\ast |B\big(p(x),\|p(x)-x\|\big)|, \quad \text{for all}~~x \in \Omega_i^\ast,
\end{equation}
cf.~\eqref{fatness1}, and define 
\begin{equation} \label{fatnessA}
	\cf^\ast:= \min_{i\in \cI_b^\ast} \hat c_i^\ast > 0.   
\end{equation}
The quantity $\cf^\ast$ is a \emph{local} indicator because $\hat c_i^\ast$ depends only on a small neighbourhood of $\Omega_i^e$.
\begin{lemma} \label{lem9}
There exists a constant ${\sf B}_3$, depending only on the local geometry parameters $\cf^\ast$ and $R_{\max}$, such that
\begin{equation}
 \|v\|_{L^2(\Omega_b^\ast)} \leq {\sf B}_3 |v|_{1,\Omega^M} \quad \text{for all}~~v \in H_0^1(\Omega^M).
\end{equation}
\end{lemma}
\begin{proof}
For $x \in \Omega_b^\ast$ we have $d(x)=\|x-p(x)\| \leq 4 R_{\max}$. Using \eqref{fatness1A} and the Hardy inequaltiy in Lemma~\ref{LemHardy} we obtain, for $v\in C_0^\infty(\Omega^M)$:
\[ |v(x)| \leq \hat c \, \cM(\|\nabla v\|)(x), \quad \hat c:=c_\cH\frac{4 R_{\max}}{\cf^\ast}.
\]
Combining this with \eqref{estimatemaximal} yields
\[
 \|v\|_{L^2(\Omega_b^\ast)} \leq \hat c\, c_{\M} |v|_{1,\Omega^M}.
\]
Application of a density argument completes the proof.
\end{proof}

Using the two lemmas above  we derive the following main result for the stability quantity $s_0$. 
\begin{theorem} \label{thmstab2}
 For the stability constant $s_0$ the bound
 \[
  s_0 \leq \sC_0
 \]
holds, with a constant $\sC_0$ that depends only on the local geometry indicators in $\cG_L^M$ and on  $N_0^\ast$, $q_{\max}, \cf^\ast$.
\end{theorem}
\begin{proof}
For $v \in H_0^1(\Omega^M)$ we define $\hat v_0= Q_0 v$, $\hat v_i:= \theta_i (v - Q_0v) \in H_0^1(\Omega_i)$, i.e., $v = \sum_{i=0}^M \hat v_i$, and $\sum_{i=1}^M \hat v_i= v - \hat v_0$.  
 Using Theorem~\ref{thm1} we obtain, along the same lines as in the proof of Theorem~\ref{thmstab}, \rev{with $\Omega_0:= \Omega^M$,}
\begin{align*} 
 s_0 & =\inf_{\sum_{j=0}^M v_j =v}\sum_{i=0}^M \Big|P_i \sum_{j=i+1}^M v_j \Big|_{1,\Omega_i}^2  
 \leq 
 \sum_{i=0}^M \Big|P_i \sum_{j=i+1}^M  \hat v_j\Big|_{1,\Omega_i}^2  
 \\  &
 \rev{ \leq |P_0(v - \hat v_0)|_{1,\Omega^M}^2 + \sum_{i=1}^M \Big|P_i\sum_{j \in \cN_i} \hat v_j \Big|_{1,\Omega_i}^2
  \leq |v-\hat v_0|_{1,\Omega^M}^2 +  \sum_{i=1}^M \Big|\sum_{j \in \cN_i} \hat v_j \Big|_{1,\Omega^M}^2} \\
  &\leq |v-\hat v_0|_{1,\Omega^M}^2 +  \rev{\Nmax^2} \sum_{i=1}^M | \hat v_i |_{1,\Omega_i}^2 \\
   &  \leq (1+ \rev{\Nmax^2} {\sC}_1)|v- Q_0v|_{1,\Omega^M}^2 + \rev{\Nmax^2} {\sC}_2 \sum_{i\in \cI_{\intset}} \|v- Q_0v\|_{L^2(\Omega_i)}^2 .
\end{align*}
Using
\[
  |v- Q_0v|_{1,\Omega^M}^2 \leq  \sum_{i \in \cI_{\intset}^\ast}|v- Q_0v|_{1,\Omega_i}^2  +  \sum_{i \in \cI_{b}^\ast}|v- Q_0v|_{1,\Omega_i}^2,
\]
and
\[
 \sum_{i\in \cI_{\intset}} \|v- Q_0v\|_{L^2(\Omega_i)}^2 \leq \sum_{i\in \cI_{\intset^\ast}} \|v- Q_0v\|_{L^2(\Omega_i)}^2 + \sum_{i\in \cI_b^\ast} \|v- Q_0v\|_{L^2(\Omega_i)}^2,
\]
we obtain, with $\tilde c:= \max\{1+ \Nmax^2 {\sC}_1, \Nmax^2 {\sC}_2 \}$:
\begin{equation}
 s_0 \leq \tilde c \, \Big( \sum_{i \in \cI_{\intset}^\ast}\|v- Q_0v\|_{1,\Omega_i}^2  +  \sum_{i \in \cI_{b}^\ast}\|v- Q_0v\|_{1,\Omega_i}^2\Big).
\end{equation}
For the term $\sum_{i\in \cI_{\intset}^\ast}$ we  use the result \eqref{RR2}:
\[
  \sum_{i\in \cI_{\intset}^\ast} \|v- Q_0v\|_{1,\Omega_i}^2 \leq {\sf B}_2^2\sum_{i\in \cI_{\intset}^\ast} |v|_{1,\Omega_i^\ast}^2 \leq {\sf B}_2^2 N_0^\ast  |v|_{1, \Omega^M}^2. 
\]
For the term $\sum_{i\in \cI_{b}^\ast}$ we  use the result \eqref{RR1}:
\[ 
\begin{split}
  \sum_{i\in \cI_{b}^\ast} \|v- Q_0v\|_{1,\Omega_i}^2 &  \leq 2 \sum_{i\in \cI_{b}^\ast} |v|_{1,\Omega_i}^2 +
  2 {\sf B}_1^2\sum_{i\in \cI_{b}^\ast} \|v\|_{L^2(\Omega_i^\ast)}^2 \\
   & \leq 2 N_0 |v|_{1, \Omega^M}^2 + 2 {\sf B}_1^2 N_0^\ast \|v\|_{L^2(\Omega_b^\ast)}^2.
  \end{split}
\]
Finally, the term $\|v\|_{L^2(\Omega_b^\ast)}$ can be estimated as in Lemma~\ref{lem9}. Combining these results completes the proof.
\end{proof}
\ \\
\begin{remark} \label{remapproximate} \rm
The Schwarz domain decomposition method  analyzed in this paper is not feasible in practice, because it iterates ``on the continuous level''. In every iteration of the method one has to solve \emph{exactly} a Poisson equation (with homogeneous Dirichlet data) on each of the balls $\Omega_i$, $i=1, \ldots,M$. If the coarse global space $V_0$ is used, this requires the exact solution of a sparse linear system with a matrix of dimension approximately $M \times M$. Note that the local problems are PDEs, whereas the global one is a sparse linear system. In practice, the exact solves on the balls  are replaced by inexact ones, which can be realized very efficiently using harmonic polynomials in terms of spherical coordinates using spherical harmonics. 
These inexact local solves define a \emph{discretization}  of the original global  PDE. If the coarse global space is used, the resulting linear system can be solved using a multilevel technique, which then requires  computational work for the coarse space correction that is  linear in $M$.  
Numerical experiments (so far only for the case without coarse global space) indicate that the method with inexact local solves has convergence properties very similar to the one with exact solves that is analyzed in this paper as the local discretization can be systematically improved. An analysis of the method with inexact local and global solves is a topic for future research. We expect that such an analysis can be done using tools and results from the literature  on subspace correction methods, because in  that framework  the effect of inexact solves on the rate of convergence of the resulting solver has been thoroughly studied, cf. e.g. \cite{Xuoverview}.
\end{remark}

\begin{remark} \label{Remadditive} \rm 
 In a setting with very large scale problems solved on parallel architectures, the \emph{additive} Schwarz domain decomposition method, also called ``parallel subspace correction method'' \cite{Xuoverview}, is (much) more efficient than the Schwarz domain decomposition method considered in this paper (``successive  subspace correction method''). We indicate how the results obtained in this paper can be applied to the setting of an additive Schwarz method. We first consider the case without a coarse  global space. In the addititve method one obtains a \emph{preconditioned system} with a symmetric positive definite operator $T:=\sum_{j=1}^M P_j$, cf.~\cite{Xu,Toselli}. The quality of the additive Schwarz preconditioner is measured using the condition number ${\rm  cond}(T)=\frac{\lambda_{\max}(T)}{\lambda_{\min}(T)}$. \rev{A (sharp) bound for this condition number is derived in \cite[Theorem 2.7]{Toselli}. This bound depends on three parameters, denoted by $\omega$, $\rho(\cE)$, $C_0$, cf. \cite{Toselli} for definitions. In the case of ``exact subspace solvers'' that we consider, we have $\omega=1$. We use a (stable) decomposition as in Theorem~\ref{thm1}. Due to $a(\theta_i v, \theta_j v)= 0 $ for all $j \notin \cN_i$ it follows that $\rho(\cE) \leq N_{\max}$. The parameter $C_0$ quantifies the stability of the decomposition. For this essentially the same bounds as for $s_0$ in Theorems~\ref{thmstab} and \ref{thmstab2} hold.} 
\end{remark}

\section{Numerical experiments} \label{Sectexperiments}
%
We only present a few numerical examples in order to complement results already available in the literature. In~\cite{ddCOSMO1} results of numerical experiments are presented which show that the method scales linearly in the number of balls in linear chains of same-sized balls in three dimensions. In~\cite{ciaramella2018analysis} results that show linear scaling are given for more general essentially one-dimensional structures, but limited to two spatial dimensions. 
Numerical results of the domain decomposition method applied to real molecules are presented in \cite{ddCOSMO2}, more precisely treating alanine chains and compounds of one to several hemoglobin units.
Further~\cite{ciaramella2019scalability} sheds light on the difference in terms of the number of iterations between chain-like versus globular biologically relevant molecules and the sensitivity with respect to the radii $R_i$. 

Here we restrict to a few theoretical scenarios, i.e. we consider the center of the balls to lie on a regular unit lattice $\mathcal L_{\bf n}$, for ${\bf n} = (n_x,n_y,n_z)$, defined by $\mathcal L_{\bf n} = [1,n_x]\times [1,n_y] \times [1,n_z] \cap \mathbb N^3$.
We take all radii  the same and equal to 0.9 in order that no inner holes appear in the structures and that there is significant overlap between neighbouring balls.
We consider three different cases:
\begin{itemize}
\item Case 1: $n_x, n_y$ are fix and $n_z=n\in\mathbb N$ is growing. This represents a lattice that is growing in one direction. 
\item Case 2: $n_x$ is fix and $n_y=n_z=n\in\mathbb N$ is growing. This represents a lattice that is growing in two directions.
\item  Case 3: $n_x=n_y=n_z=n\in\mathbb N$ is growing. This represents a lattice that is growing in all three directions.
\end{itemize}
Case 3 differs from the first two cases in the sense that for the global fatness indicator $d_F$ we have, in the limit $n\to\infty$, $d_F \sim  n$ in former case, whereas for Case 1 and 2, $d_F$ is uniformly bounded (with respect to $n$), but depends on $n_x, n_y$ and $n_x$, respectively. 

\rev{
Since the contraction behaviour only depends on the operator, we consider the simple setting of $f=0$, i.e. $-\Delta u =0$ with trivial solution, and observe the contraction of a given non-trival starting function $u^0$ to $u\equiv 0$. Indeed,  $u^0$ is chosen equal to one in the inner balls of the domain and  smoothly decreases to zero boundary values within the boundary layer of balls.
The local problems are solved using spherical harmonics times a radial monomial. Indeed, since we know a priori that the solution is harmonic we use a spectral method with harmonic basis functions of the form $r^k \, Y_{km}(\theta,\varphi)$, where $(r,\theta,\varphi)$ denote the spherical coordinates with respect to the balls center and $Y_{km}$ the real-valued spherical harmonics. We choose the degree of spherical harmonics high enough such that the approximation error is not affecting the iterative process.

Figure \ref{fig:Lattice} illustrates the number of iterations of the algorithm presented in~\eqref{SSC} with and without coarse correction to reach a relative error criterion $|u^\ell|_{1,\Omega^M} < 10^{-6} \, |u^0|_{1,\Omega^M}$.
As predicted by the theory presented in this paper, we have a bounded number of iterations in Case 1 and 2 (with dependency on $n_x$ (and $n_y$ in Case 2), while Case 3 shows a growing number of iterations that is approximately linear in $n$ if no coarse global space correction is employed. 
On the other hand, all cases show a bounded number of iterations, as predicted by the theory, if the coarse global space correction is employed.
}

\begin{figure}[t!]
	\centering
	\includegraphics[trim=0pt 10pt 0pt 0pt, clip,scale=0.3]{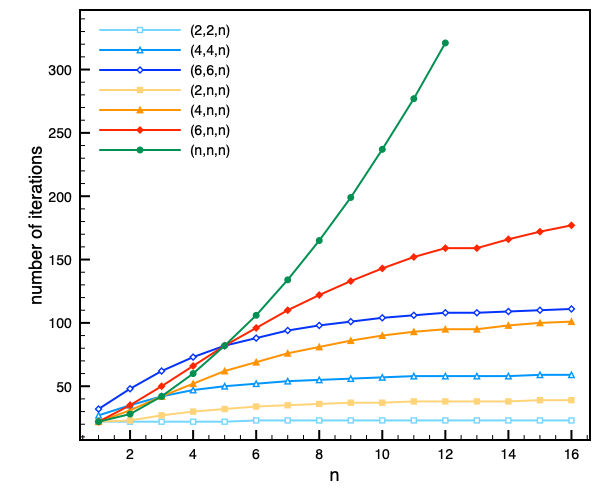}
	\includegraphics[trim=0pt 10pt 0pt 0pt, clip,scale=0.3]{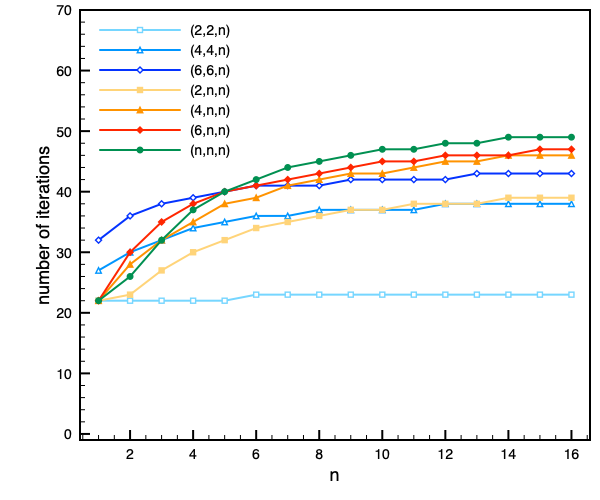}
	\caption{Number of iterations of the Algorithm~\eqref{SSC} to reach a fixed error tolerance for different lattice-structures labelled by the description $(n_x,n_y,n_z)$ without (left) and with (right) coarse correction.}
	\label{fig:Lattice}
\end{figure}

\section{Conclusion and outlook}
In this paper we presented an analysis of the rate of convergence of the overlapping Schwarz domain decomposition method applied to the Poisson problem on a special family of domains, motivated by applications in implicit solvation models in computational chemistry. The analysis uses the framework of subspace correction methods, in which the contraction number of the error propagation operator (in the natural energy norm) can be expressed in only one stability parameter, namely $s_0$ appearing in Lemma~\ref{lemmaXu}. 
We investigate how this stability parameter depends on relevant geometric properties of the domains, such as the amount of overlap between neighbouring balls and the exterior fatness of the domain. The latter plays a key role in the analysis of a Poincar\'e inequality. To formalize the dependence of the constants in the estimates on relevant geometric properties we introduced local geometry indicators, cf.~\eqref{locgeo}, and a global geometry indicator $d_F$ \eqref{indDf}. In view of our applications it is reasonable to assume that the local indicators are uniformly (in the number of subdomains) bounded. We derive bounds for the stability parameter $s_0$ as presented in the main Theorem~\ref{thmstab}. 
The results of this theorem 
show that the rate of convergence of the overlapping Schwarz domain decomposition method can deteriorate in situations with very large $d_F$ values and remains constant for (possibly complex and non-trivial) geometrical structures where $d_F$ remains constant, which is the case in the majority of cases when dealing with biochemical molecules. This result provides therefore a theoretical justification of the performance of the solvers observed in implicit solvation models.
In cases where $d_F$ is large, the efficiency of the method can be significantly improved by using an additional coarse global space, as explained in section~\ref{sectcoarse}. Including such a space results in a uniform bound for $s_0$ that in particular does not depend on $d_F$, cf. Theorem~\ref{thmstab2}.

In this paper we restricted the analysis to the case with exact subspace solvers, both for  the local Poisson problems in $H_0^1(\Omega_i)$ and the Poisson problem in the coarse global space $V_0$. In future work we want to study the effect of inexact solvers by means of a local discretization error. In recent years the method without the coarse global space has been used for the efficient simulation of many complex applications in the field of implicit solvation models. So far we did not perform a systematic numerical study of the method with the coarse global space. We plan to do this in the near future. 

\section{Appendix} 
 In this appendix we  give a proof of Lemma~\ref{lem:DistEq} and  a derivation of the result stated in Remark~\ref{remsmoothness}.
\subsection{Proof of Lemma~\ref{lem:DistEq}} 
\label{App1} 
\begin{proof}
Define $\alpha=\frac{\pi}{2}-\beta^\infty$, where $\beta^\infty$ is the minimal angle (half of the aperture) constant of Assumption {\bf (A4)}, and
\begin{align*}
	\varepsilon := R^\infty_{\rm min}(1-\sin(\alpha)).
\end{align*}
We introduce the splitting of $\Omega_{\bset}$ into 
\begin{align*}
	\Omega_\varepsilon &= \{ x \in \Omega_{\bset}  \;|\; \delta(x) < \varepsilon \}, \qquad
	\Omega_\varepsilon^{\rm c} = \{ x \in \Omega_{\bset}  \;|\; \delta(x) \ge \varepsilon \}.
\end{align*}
First, for 
  $x\in \Omega_\varepsilon^{\rm c}$ we have
$
	\frac{\delta(x)}{{\rm dist}(x,  \partial \Omega^M)}	
	\ge  
	 \frac{\varepsilon}{2R^\infty_{\max}}
$.
Hence, 
\begin{equation} \label{Est1}
	 	\delta(x) \geq  \tfrac{R^\infty_{\rm min}(1-\sin(\alpha))}{2R^\infty_{\max}}\,{\rm dist}(x,  \partial \Omega^M) \quad \text{for all}~x \in \Omega_\varepsilon^{\rm c}.
\end{equation} 
Second, we consider $x\in \Omega_\varepsilon$. 
We start with some preliminary consideration.
We denote by $\sphericalangle(v,w) := \arccos\left(\frac{v\cdot w}{\|v\|\|w\|}\right)$ the angle between two vectors $v,w\in\R^3$.
The circular cone with apex $y\in\R^3$, axis $w\in\R^3$ and aperture $2\alpha$ is denoted by
\[
	K_y(w,\alpha) = \Big\{ z = y + v \in \R^3 \;|\; v\in\R^3, \sphericalangle(w,v) \le \alpha \Big\}.
\]
%
Let  $y = p(x)\in\partial\Omega^M$ denote one of the closest points of $x$ on $\partial \Omega^M$. 
Following the notation introduced in~\cite{Voronoi}, let $\mathbf i =  \mathcal I(y) = \{ i_1,\ldots,i_r\}$ be the maximal set of indices such that 
$
	y\in \bigcap_{i\in\mathcal I(y)} \partial \Omega_i.
$

If $\mathbf i  = \mathcal I(y)=\{i_1\}$, i.e., if $y$ is contained on a spherical patch and only belonging to one sphere, then there  trivially holds
\begin{equation} \label{Est2}
	\delta(x) \ge \delta_{i_1}(x) = \| x - p(x) \| = {\rm dist}(x,\partial \Omega^M).
\end{equation}

Otherwise, define $\mathbf m_{\mathbf i}  = \{m_{i_1},\ldots,m_{i_r}\}$, i.e., the set of the centers of all spheres that contain $y$, and introduce the generalized cone
\begin{align*}
	\mbox{cone}_y(\mathbf m_{\mathbf i}) &:= \left\{ w = y + \sum_{t=1}^r\lambda_t v_{t}  \;\middle|\; 0 \le \lambda_t\right\}, 
	\qquad 
	v_{t} := \frac{m_{i_t}-y}{\|m_{i_t}-y\|}.
\end{align*}
Then, following~\cite{Voronoi}[Theorem 1], there holds that 
$x \in \mbox{cone}_y(\mathbf m_{\mathbf i})$.
We now show that we can cover $\mbox{cone}_y(\mathbf m_{\mathbf i})$ with $r$ circular cones with equal aperture.
Indeed, considering the vector $-n(y)$ from Assumption {\bf (A4)} we have
\[
	-n(y)\cdot v_{t}>\gamma_\alpha^\infty = \cos(\tfrac{\pi}{2}- \beta^\infty)=\cos(\alpha), \qquad \forall t=1,\ldots,r.
\]
Elementary geometrical considerations then yield that we can cover  $\mbox{cone}_y(\mathbf m_{\mathbf i})$ with circular cones $K_y(v_{t},\alpha)$, $t=1,\ldots,r$, of angle $\alpha$, i.e., 
\[
	\mbox{cone}_y(\mathbf m_{\mathbf i}) \subset \bigcup_{t=1}^r K_y(v_{t},\alpha),
\]
see~Figure~\ref{fig:ConeLem3p1} (right) for an elementary illustration. 
Thus there exists a $j=i_s$, such that 
$
	x \in K_y(v_{s},\alpha)
$
holds. 
Then, choose the maximal radius $R_{j,\alpha} = R_j\sin(\alpha)>0$ such that $B(m_j;R_{j,\alpha})$ is entirely contained in the circular cone $K_y(v_{s},\alpha)$; 
see Figure~\ref{fig:ConeLem3p1} for a graphical illustration of the geometric situation.
\rev{
Due to $x \in \Omega_\varepsilon$ we have $\delta_j(x) \leq \delta(x)  < \varepsilon$. 
First, this implies that
\[
	R_j -  \| x -  m_j \| < \varepsilon = R^\infty_{\min}(1-\sin(\alpha)) \le R_j(1-\sin(\alpha)),
\]
and $\| x -  m_j \| > R_j \sin(\alpha) = R_{j,\alpha}$. From this we conclude
\[
	 x \in K_y(v_{j},\alpha) \cap \big( B(m_j;R_j) \setminus \overline{B(m_j;R_{j,\alpha})} \big).
\]
Second, if $\| x- y \| \le R^\infty_{\min} \sin(\beta^\infty) = R^\infty_{\min} \cos(\alpha)$, then $x\in \overline{B(y; R_j\cos(\alpha))}$ and 
%
%
%
Lemma~\ref{lem:calpha} (below) then states that 

\begin{equation}
	\label{eq:EquivShortDist}
	\delta_j(x) \ge \tfrac{\cos(\alpha)}{2} \, {\rm dist}(x,\partial\Omega^M). 
\end{equation}
}
In consequence, equation~\eqref{eq:EquivShortDist} yields
that
\[
	\tfrac{\cos(\alpha)}{2}  \,  {\rm dist}(x,\partial\Omega^M) 
	= \tfrac{\cos(\alpha)}{2} \, {\rm dist}(x,y)
	\le \delta_j(x) 
	\le \delta(x).
\]
\rev{
In the contrary case, i.e. if $\| x- y \| > R^\infty_{\min} \sin(\beta^\infty)$, then
\[
	{\rm dist}(x,\partial\Omega^M) \le 2 R_j \le 2 R^\infty_{\max} \le  \tfrac{2 R^\infty_{\max}}{\gint} \delta(x),
\]
by \eqref{overlap1}.
}
Finally, define (cf. also \eqref{Est1}, \eqref{Est2})
\begin{equation} 
	\label{formg}
	\gamma_\bset = \min\left\{1,\tfrac{R^\infty_{\rm min}(1-\sin(\alpha))}{2R^\infty_{\max}},\tfrac{\cos(\alpha)}{2}\rev{,\tfrac{\gint}{2 R^\infty_{\max}}}\right\},
\end{equation}

and we obtain $\delta(x) \geq \gamma_\bset \, {\rm dist}(x,\partial\Omega^M)$.
\begin{figure}[t!]
	\centering
	\includegraphics[trim=50pt 110pt 480pt 110pt, clip, scale=0.35]{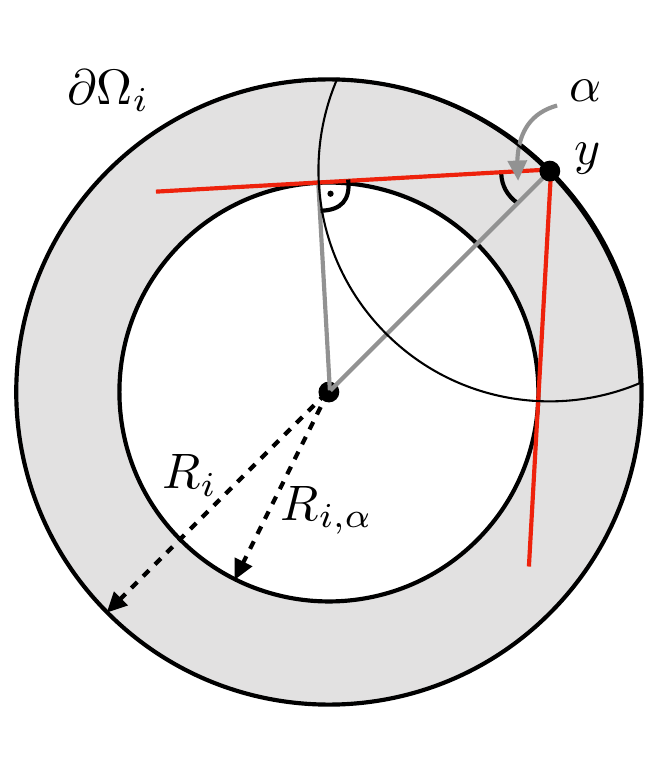}
	\hspace{30pt}
	\includegraphics[trim=50pt 110pt 290pt 110pt, clip, scale=0.35]{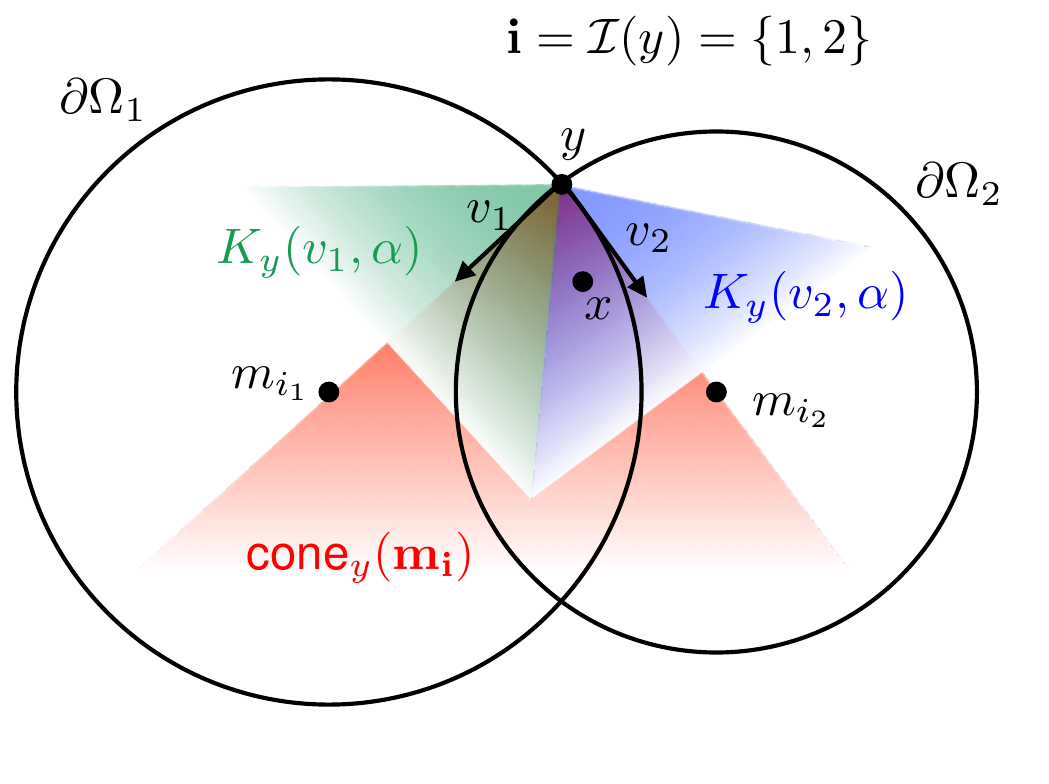}
\caption{Illustration on choosing the maximal radius $R_{i,\alpha}$ (left) and how to cover $\mbox{cone}_y(\mathbf m_{\mathbf i})$ with circular cones $K_y(v_{t},\alpha)$ (right).}
	\label{fig:ConeLem3p1}
\end{figure}
\end{proof}


\begin{lemma}
\label{lem:calpha}
For $y \in \partial \Omega_i$ define $v_{i} := \frac{m_{i}-y}{\|m_{i}-y\|}$.
For given $0<\alpha<\frac{\pi}2 $, consider the circular cone $K_y(v_i,\alpha)$ and 
 let $R_{i,\alpha}>0$ be the minimal radius such that $B(m_i;R_{i,\alpha})$ intersects the boundary $\partial K_y(v_i,\alpha)$, i.e. the maximal radius such that $B(m_i;R_{i,\alpha})$ is contained in the circular cone $K_y(v_i,\alpha)$. 
Then, the maximal value is given by $R_{i,\alpha} = R_i\sin(\alpha)$ and the following holds:
\begin{equation}
	\label{eq:EquivShortDist2}
	\delta_i(x) \ge c_\alpha {\rm dist}(x,y), \qquad \forall x \in K_y(v_i,\alpha)\cap \big( B(m_i;R_i) \setminus \overline{B(m_i;R_{i,\alpha})} \big)\rev{\cap \overline{B(y; R_i\cos(\alpha))}},
\end{equation}
where $c_\alpha = \frac{\cos(\alpha)}{2}>0$.
\end{lemma}
\begin{proof}
%
We consider any point $x\in \mathcal P = K_y(v_i,\alpha)\cap \big( B(m_i;R_i) \setminus \overline{B(m_i;R_{i,\alpha})} \big)\rev{\cap \overline{B(y; R_i\cos(\alpha))}}$. 
Consider now all points in $\mathcal P$ lying on a subset $\Gamma(x)$ of the sphere with radius $R_i-\delta_i(x)$:
\[
	\Gamma(x) = \{ z\in \mathcal P \;|\; \delta_i(z) = \delta_i(x) \},
\]
and introduce $d_i(x) = \max_{z\in \Gamma(x)} {\rm dist}(z,y)$ such that 
$
	\delta_i(x) \le {\rm dist}(x,y) \le d_i(x),
$
and thus
$
	\frac{ {\rm dist}(x,y)}{\delta_i(x)} \le \frac{ d_i(x)}{\delta_i(x)}.
$
Applying the cosine rule yields
$
	(R_i-\delta_i(x))^2 = d_i(x)^2 + R_i^2 - 2 d_i(x) R_i \cos(\alpha),
$
which is equivalent to
\[
	\frac{d_i(x)}{ \delta_i(x)} = \frac{2R_i - \delta_i(x)}{2R_i\cos(\alpha) - d_i(x)}.
\]
%
Since $d_i(x)\le d_i=R_i\cos(\alpha)$, we have
\[
 \frac{\delta_i(x)}{{\rm dist}(x,y)} \geq \frac{\delta_i(x)}{d_i(x)} =  \frac{2R_i\cos(\alpha) - d_i(x)}{2R_i - \delta_i(x)} \geq \frac{R_i\cos(\alpha) }{2R_i - \delta_i(x)} \geq \frac{ \cos (\alpha)}{2},
\]
which completes the proof.
\end{proof}

\subsection{Derivation of result in Remark~\ref{remsmoothness}} 
\label{App2}
We consider the case of two overlapping balls , $\Omega^2= \Omega_1 \cap \Omega_2$, with $\Omega_1=B\big((-1,0,0);2\big)$, $\Omega_2=B\big((1,0,0);2\big)$ as in Remark~\ref{remsmoothness}. The intersection of $\partial \Omega^2$ with $\Omega_1 \cap \Omega_2$ is given by the circle $S=\left\{ (0,x_2,x_3)\; \middle| \; x_2^2+x_3^2=3\right\}$. We analyze the smoothness of $\theta_1$ (close to $S$). Elementary computation yields
\[
  \nabla \theta_1= \frac{\delta_2 \nabla \delta_1- \delta_1 \nabla \delta_2}{\delta^2}, \quad \nabla\theta_1 \cdot \nabla \theta_1= \frac{\delta_1^2 +\delta_2^2 - 2 \delta_1 \delta_2 \nabla\delta_1 \cdot \nabla\delta_2}{\delta^4}.
\]
On the subdomain $V:=\left\{ x \in \Omega_1 \cap \Omega_2\; \middle| \; |x_1| \leq \frac12,~1 \leq x_2^2+ x_3^2 \leq 3\right\}$ we have $|\nabla\delta_1 \cdot \nabla\delta_2| \leq \frac34$. Hence,
\[
  (\nabla\theta_1 \cdot \nabla \theta_1)_{|V} \geq \frac{\delta_1^2 +\delta_2^2 -  1\frac12 \delta_1 \delta_2}{\delta^4} \geq \frac14 \frac{\delta_1^2 +\delta_2^2}{\delta^4} \geq \frac18 \frac{1}{\delta^2}.
\]
Take $p$ on the intersection circle $S$ and define the triangle  $T_p$ with  the vertices $p$, $(-1,0,0)$, $(0,0,0)$. For all $x \in  T_p$ we have $\delta_2(x) \leq \delta_1(x)$, hence $\delta(x) \leq 2 \delta_1(x)$. Furthermore, there exists a constant $c$ such that $\delta_1(x) \leq c \|x-p\|$ for all $x \in T_p \cap V$. The spherical sector obtained by rotating $T_p$ along $p \in S$ can be parametrized by coordinates $(s,\rho,\theta)$, with $s \in [0,2 \sqrt{3} \pi]$ the arclength parameter on $S$ and $(\rho,\theta)$ polar coordinates in the triangle $T_p$  at $p=s$, with origin at $p$. Integration over a part $T_S:=\{\, (s,\rho,\theta)~|~\rho\leq \rho_0\,\}$ of this spherical sector, with $\rho_0 >0$ sufficiently small, yields
\[
  \int_{T_S} |\nabla \theta_1 |^2 \, dx \geq \frac{1}{16} \int_{ T_S} \delta_1^{-2} \, dx \geq  \tilde c \int_0^{2 \sqrt{3} \pi} \int_0^{\frac16 \pi} \int_{0}^{\rho_0} \frac{1}{\rho^2} \rho ~ d\rho \, d\theta \, ds = \infty.
\]

\bibliographystyle{siam}
\bibliography{literatur}{}

\end{document}